\documentclass{amsart}

\usepackage{amssymb}
\usepackage{enumitem}
\usepackage{graphicx}

\newtheorem{theorem}{Theorem}[section]
\newtheorem{corollary}[theorem]{Corollary}
\newtheorem{lemma}[theorem]{Lemma}

\theoremstyle{definition}
\newtheorem{definition}[theorem]{Definition}
\newtheorem{remark}[theorem]{Remark}
\newtheorem{example}[theorem]{Example}

\usepackage{color}

\numberwithin{equation}{section}

\begin{document}

\thispagestyle{empty}

\title[Inner Riesz balayage and applications]{A theory of inner Riesz balayage and its applications}

\author{Natalia Zorii}
\address{\footnotesize{Institute of Mathematics\\ Academy of Sciences
of Ukraine\\
Teresh\-chen\-kivska~3\\
01601 Kyiv, Ukraine}}
\email{\footnotesize{natalia.zorii@gmail.com}}


\date{}

\begin{abstract} We establish the theory of balayage for the Riesz kernel $|x-y|^{\alpha-n}$, $\alpha\in(0,2]$, on $\mathbb R^n$, $n\geqslant3$, alternative to that suggested in the book by Landkof. A need for that is caused by the fact that the balayage in that book is defined by means of the integral representation, which, however, so far is not completely justified. Our alternative approach is mainly based on Cartan's ideas concerning inner balayage, formulated by him for the Newtonian kernel. Applying the theory of inner Riesz balayage thereby developed, we obtain a number of criteria for the existence of an inner equilibrium measure $\gamma_A$ for $A\subset\mathbb R^n$ arbitrary, in particular given in terms of the total mass of the inner swept measure $\mu^A$ with $\mu$ suitably chosen. For example, $\gamma_A$ exists if and only if $\varepsilon^{A^*}\ne\varepsilon$, where $\varepsilon$ is a Dirac measure at $x=0$ and $A^*$ the inverse of $A$ relative to the sphere $|x|=1$, which leads to a Wiener type criterion of inner $\alpha$-ir\-reg\-ularity. The results obtained are illustrated by examples.
\end{abstract}



\maketitle

\section{Introduction}\label{sec:intr}
A major goal of our study is to establish the theory of balayage for the Riesz kernel $|x-y|^{\alpha-n}$, $\alpha\in(0,2]$, on $\mathbb R^n$, $n\geqslant3$, alternative to that suggested in \cite{L} (see Chapter~IV, Section~6, n$^\circ$~25 as well as Chapter~V, Section~1, n$^\circ$~2). A need for that is caused by the fact that the balayage $\mu^A$ of a positive Radon measure $\mu$ on $\mathbb R^n$ to a Borel set $A\subset\mathbb R^n$ is defined in \cite{L} by means of the integral representation
\begin{equation}\label{L-repr}\mu^A=\int\varepsilon_y^A\,d\mu(y),\end{equation}
where $\varepsilon_y$ is the unit Dirac measure at $y\in\mathbb R^n$.
However, this requires that the family $(\varepsilon_y^A)_{y\in\mathbb R^n}$ be $\mu$\emph{-ad\-eq\-uate\/} in the sense of \cite[Chapter~V, Section~3, n$^\circ$~1, Definition~1]{B2}. As pointed out in \cite[Chapter~V, Section~3, n$^\circ$~1, Remark]{B2}, it is not enough to verify that for every $f\in C_0(\mathbb R^n)$, the function $y\mapsto\int f\,d\varepsilon_y^A$ is $\mu$-meas\-urable on $\mathbb R^n$ (as it is done in \cite[p.~214, footnote~12]{L}); see also counterexamples (without $\mu$-adeq\-uacy) in \cite[Chapter~V, Section~3, Exercises~1, 2]{B2}. Here $C_0(\mathbb R^n)$ is the class of all finitely continuous functions on $\mathbb R^n$ with compact support.

For $A$ closed and $\mu$ carried by $A^c:=\mathbb R^n\setminus A$, the $\mu$-adequacy of the family $(\varepsilon_y^A)_{y\in A^c}$, and hence the validity of the integral representation (\ref{L-repr}), has been proven in recent work \cite[Lemma~3.16, Theorem~3.17]{FZ}. Compare with our Theorem~\ref{th-int-rep} strengthening \cite[Theorem~3.17]{FZ} to $A$ arbitrary and $\mu$ carried by $\overline{A}^{\,c}$, the concept of balayage being now understood in the sense described in Section~\ref{sec-bal} below. However, the question whether the integral representation (\ref{L-repr}) holds for \emph{any\/} $\mu$ is still open.

Leaving aside of the main stream of this study the question on the validity of the integral representation (\ref{L-repr}), we establish instead an alternative theory of Riesz balayage by generalizing H.~Cartan's ideas concerning inner balayage, formulated by him for the Newtonian kernel $|x-y|^{2-n}$ \cite{Ca2}.
To explain briefly the results thereby obtained, we need the following notions and notation.

Let $\mathfrak M$ denote the linear space of all (signed) Radon measures $\nu$ on $\mathbb R^n$, equipped with the \emph{vague\/} topology, i.e.\ the topology of pointwise convergence on the class $C_0(\mathbb R^n)$. Given $\nu,\nu_1\in\mathfrak M$, we define the \emph{potential\/} and the \emph{mutual energy\/} by
\begin{align*}U^\nu(x)&:=\int|x-y|^{\alpha-n}\,d\nu(y),\quad x\in\mathbb R^n,\\
E(\nu,\nu_1)&:=\int|x-y|^{\alpha-n}\,d(\nu\otimes\nu_1)(x,y),
\end{align*}
respectively (provided, of course, that the corresponding value on the right-hand side is well defined as a finite number or $\pm\infty$). For $\nu=\nu_1$, $E(\nu,\nu_1)$ defines the \emph{energy\/} $E(\nu):=E(\nu,\nu)$ of $\nu$. All $\nu\in\mathfrak M$ with finite $E(\nu)$ form a pre-Hil\-bert space $\mathcal E$ with the inner product $(\nu,\nu_1):=E(\nu,\nu_1)$ and the norm $\|\nu\|:=\sqrt{E(\nu)}$. The topology on $\mathcal E$ defined by $\|\cdot\|$ is said to be \emph{strong}.

For an arbitrary set $Q\subset\mathbb R^n$, we denote by $\mathfrak M^+_Q$ the cone of all positive $\nu\in\mathfrak M$ \emph{carried by\/} $Q$, which means that $Q$ is $\nu$-meas\-ur\-able and $Q^c$ is $\nu$-neg\-lig\-ible. Write $\mathcal E^+_Q:=\mathcal E\cap\mathfrak M^+_Q$, $\mathfrak M^+:=\mathfrak M^+_{\mathbb R^n}$, and $\mathcal E^+:=\mathcal E^+_{\mathbb R^n}$.

To establish the theory of inner Riesz balayage, we first consider $\mu$ with finite energy, and we define the inner balayage $\mu^A\in\mathcal E^+$ of $\mu\in\mathcal E^+$ to $A\subset\mathbb R^n$ arbitrary as the strong and vague limit of $\mu^K$ as $K$ increases along the upper directed family $\mathfrak C_A$ of all compact subsets of $A$ (see Theorem~\ref{th-f-cont}), where $\mu^K$ is the orthogonal projection of $\mu$ in the pre-Hil\-bert space $\mathcal E$ onto the strongly complete convex cone $\mathcal E^+_K$ (cf.\ Theorem~\ref{th-bala-f}). Alternatively, this $\mu^A$ is, in fact, the orthogonal projection of $\mu$ onto the strong cluster set of $\mathcal E^+_A$ (see Theorem~\ref{th-f-cont}).

Having observed that this $\mu^A$ can also be determined uniquely by means of
\[E(\mu,\lambda^A)=E(\mu^A,\lambda)\text{ \ for all\ }\lambda\in\mathcal E^+,\]
we now define the \emph{inner balayage\/} $\mu^A\in\mathfrak M^+$ of $\mu\in\mathfrak M^+$ to $A$ as a (unique) measure satisfying this symmetry identity.\footnote{The term `\emph{inner\/} balayage' is justified by showing that $\mu^K\to\mu^A$ vaguely as $K$ increases along $\mathfrak C_A$, see Theorem~\ref{incr}.}  This $\mu^A$ exists, and it can equivalently be determined by either of the following two limit relations (see Theorem~\ref{th-bala-2} and its proof):
\[\mu_k^A\to\mu^A\text{ \ vaguely},\quad U^{\mu_k^A}\uparrow U^{\mu^A}\text{ \ pointwise on\ }\mathbb R^n,\]
where $(\mu_k)\subset\mathcal E^+$ is any given sequence (net) such that $U^{\mu_k}\uparrow U^\mu$ pointwise on $\mathbb R^n$.

Note that although $U^{\mu^A}=U^\mu$ n.e.\ on $A$, that is, everywhere on $A$ except for a set of zero inner capacity $c_\alpha(\cdot)$ (see Theorem~\ref{th-bala-2}), this property no longer characterizes $\mu^A$ uniquely (as it would do for $A$ closed and $\mu\in\mathcal E^+$, cf.\ Theorem~\ref{th-bala-f}). This uniqueness nevertheless does hold whenever $A$ is closed and $\mu\in\mathfrak M^+$ is carried by $A^c$ (see Corollary~\ref{unicity}).

Further, we apply the concept of inner Riesz balayage thus introduced to the problem on the existence of an \emph{inner equilibrium measure\/} $\gamma_A$ for $A$ arbitrary. This $\gamma_A$ can be defined e.g.\ as the vague limit of $\gamma_K$ as $K$ increases along $\mathfrak C_A$, while the equilibrium measure $\gamma_K$ on $K$ compact is defined as usual (see e.g.\ \cite[Chapter~II, Section~2, n$^\circ$~7]{L}).

For $A$ Borel, necessary and sufficient conditions for the existence of $\gamma_A$ have been provided in \cite[Theorem~5.1]{L}. However, \cite[Theorem~5.1]{L} has not been completely justified, because the proof of its necessity part is based on the concept of balayage, introduced in \cite[Chapter~IV, Section~6, n$^\circ$~25]{L} with the aid of the integral representation (\ref{L-repr}).

By use of our concept of inner Riesz balayage, we fix that gap in \cite[Proof of Theorem~5.1]{L} and, moreover, we strengthen \cite[Theorem~5.1]{L} to $A$ arbitrary (see Theorem~\ref{l-eq-ex}). Having observed that the existence of $\gamma_A$ does not necessarily imply the finiteness for $c_\alpha(A)$, we illustrate this by means of Example~\ref{ex-thin}.

An \emph{inner\/} $\alpha$\emph{-irregular\/} point $y$ for $A$ is defined by the relations $y\in\overline{A}$ and $\varepsilon^A_y\ne\varepsilon_y$. We show that $\varepsilon^A_y\ne\varepsilon_y$ is equivalent to the existence of an inner equilibrium measure $\gamma_{A^*}$ for the inverse $A^*$ of $A\setminus\{y\}$ relative to the sphere $S(y,1)$; and then $\varepsilon^A_y$ is, actually, the Kelvin transform of $\gamma_{A^*}$ (Theorem~\ref{nands}). Combining Theorem~\ref{nands} with the above-men\-tion\-ed Theorem~\ref{l-eq-ex} results in a Wiener type criterion of inner $\alpha$-ir\-reg\-ul\-arity (Theorem~\ref{Wiener}).

Other necessary and sufficient conditions for the existence of $\gamma_A$, now given in terms of the total mass $\mu^A(\mathbb R^n)$ with $\mu$ suitably chosen, are provided by Theorems~\ref{bal-mass-th} and \ref{bal-mass-th'}. It is shown that $\gamma_A$ exists whenever there is $\mu\in\mathfrak M^+_{\,\overline{A}^{\,c}}$ with 
\[\mu^A(\mathbb R^n)<\mu(\mathbb R^n),\] 
while for closed $A$, this can be reversed (see Section~\ref{fur-cr} for further details). The proofs of Theorems~\ref{bal-mass-th} and \ref{bal-mass-th'} are based on the above-men\-tion\-ed Theorem~\ref{nands} as well as on Theorem~\ref{th-int-rep}, establishing the integral representation (\ref{L-repr}) for $A$ arbitrary and $\mu$ carried by $\overline{A}^{\,c}$ and thereby strengthening \cite[Theorem~3.17]{FZ}. Theorems~\ref{bal-mass-th} and \ref{bal-mass-th'} are illustrated by Example~\ref{ex2}.

It is worth mentioning that the concept of inner Riesz balayage, defined in our study, differs from the concept of balayage by Brelot \cite{Br} as well as from that by Bliedtner and Hansen \cite{BH}.
The last two are relevant to the concept of \emph{outer\/} balayage (cf.\ \cite{Ca2}), which can be seen e.g.\ by comparing \cite[Theorem~IX.10]{Br} and \cite[Chapter~VI, Proposition~2.2]{BH} with our Theorem~\ref{Wiener} and Corollary~\ref{cor-decr}, respectively.

\section{Preliminaries}\label{sec:prel}
This paper deals with the Riesz kernel $|x-y|^{\alpha-n}$ of order $0<\alpha\leqslant2$ on $\mathbb R^n$, $n\geqslant3$.
In what follows we shall tacitly use the notions and notation introduced in Section~\ref{sec:intr}.

For $Q\subset\mathbb R^n$, let $\partial Q$ and $\overline{Q}$ denote the boundary and the closure of $Q$ in the Euclidean topology on $\mathbb R^n$. Write $B(y,r):=\{x\in\mathbb R^n: \ |x-y|<r\}$, where $r>0$, and let $S(y,r)$ and $\overline{B}(y,r)$ stand for the boundary and the closure of $B(y,r)$ in $\mathbb R^n$.

In this section we have compiled some basic facts in Riesz potential theory, often used below.
When speaking of a measure $\mu\in\mathfrak M^+$, we always tacitly assume that $U^\mu$ is not identically infinite, or equivalently \cite[Chapter~I, Section~3, n$^\circ$~7]{L}
\begin{equation}\label{1.3.10}\int_{|y|>1}\,\frac{d\mu(y)}{|y|^{n-\alpha}}<\infty.\end{equation}
Then $U^\mu$, $\mu\in\mathfrak M^+$, is $\alpha$-\emph{sup\-er\-har\-monic\/} (hence, lower semicontinuous (l.s.c.)) on $\mathbb R^n$ \cite[Chapter~I, Section~6, n$^\circ$~20]{L}, which is crucial to Theorems~\ref{3.9} and \ref{dom} below.

\begin{theorem}\label{3.9}If an upper directed family\/ $(U^{\mu_t})_{t\in T}$, where\/ $\mu_t\in\mathfrak M^+$ for all\/ $t\in T$, is majorized by\/ $U^\mu$ with some\/ $\mu\in\mathfrak M^+$, then there exists\/ $\nu\in\mathfrak M^+$ such that\/ $U^{\mu_t}\uparrow U^\nu$ pointwise on\/ $\mathbb R^n$ and\/ $\mu_t\to\nu$ vaguely\/ {\rm(}as\/ $t$ increases along\/ $T$\rm{)}.
\end{theorem}

For $T$ countable, Theorem~\ref{3.9} is \cite[Theorem~3.9]{L}. The proof of \cite[Theorem~3.9]{L} can be generalized to $T$ uncountable with the aid of \cite[Appendix~VIII, Theorem~2]{Doob} and \cite[Chapter~IV, Section~1, Theorem~1]{B2}.

The Riesz kernel is \emph{strictly positive definite}, that is, $E(\nu)\geqslant0$ for every $\nu\in\mathfrak M$ (whenever $E(\nu)$ is well defined) and $E(\nu)=0$ only for $\nu=0$.
Furthermore, it is \emph{perfect\/} \cite{F1} in the sense that every strong Cauchy sequence (net) in $\mathcal E^+$ converges strongly to any of its vague cluster points, and the strong topology on $\mathcal E^+$ is finer (stronger) than the vague topology on $\mathcal E^+$. Since any strongly bounded part of $\mathcal E^+$ is vaguely bounded \cite[Lemma~2.5.1]{F1}, the cone $\mathcal E^+$ is strongly complete. Hence, so is $\mathcal E^+_F$ for $F\subset\mathbb R^n$ closed, the cone $\mathfrak M^+_F$ being vaguely closed.

For any $Q\subset\mathbb R^n$, the \emph{inner\/} $\alpha$-Riesz \emph{capacity\/} $c_\alpha(Q)$ is given by\footnote{The infimum over the empty set is taken $+\infty$. We also put $1/(+\infty)=0$ and $1/0=+\infty$.}
\[c_\alpha(Q):=1\bigl/\inf\,E(\mu),\]
where $\mu$ ranges over all $\mu\in\mathcal E^+_Q$ with $\mu(\mathbb R^n)=1$. Then
\begin{equation}\label{cap-incr}c_\alpha(K)\uparrow c_\alpha(Q)\text{ \ as\ }K\uparrow Q,\end{equation}
where the abbreviation `$K\uparrow Q$' means that $K$ increases along $\mathfrak C_Q$.

\begin{lemma}[{\rm see \cite[Lemma~2.3.1]{F1}}]\label{2.3.1}For any\/ $Q\subset\mathbb R^n$, $c_\alpha(Q)=0\iff\mathcal E^+_Q=\{0\}$.\end{lemma}

A measure $\mu\in\mathfrak M^+$ is said to be \emph{bounded\/} if $\mu(\mathbb R^n)<\infty$, and \emph{absolutely continuous\/} if $\mu(K)=0$ for every compact $K\subset\mathbb R^n$ with $c_\alpha(K)=0$. It follows from Lemma~\ref{2.3.1} that any $\mu\in\mathcal E^+$ is absolutely continuous; but not conversely \cite[pp.~134--135]{L}.

An assertion $\mathcal U$ is said to hold \emph{nearly everywhere\/} (\emph{n.e.}) on $Q\subset\mathbb R^n$ if the set of all $x\in Q$ for which $\mathcal U(x)$ fails has zero inner capacity.

The following assertion amounts to the countable subadditivity of inner capacity in the form stated in \cite[p.~158, Remark]{F1}.

\begin{lemma}\label{subad}Let\/ $Q\subset\mathbb R^n$ be arbitrary, and\/ $E_k\subset\mathbb R^n$, $k\in\mathbb N$, Borel. If an assertion\/ $\mathcal U$ holds n.e.\ on $Q\cap E_k$ for every\/ $k$, then\/ $\mathcal U$ holds n.e.\ on the union of all\/ $Q\cap E_k$.\end{lemma}

The property of the Riesz kernel of order $\alpha\in(0,2]$, presented in the following assertion (see \cite[Theorems~1.27, 1.29]{L}), is known as the \emph{complete maximum principle\/}; for $q=0$, it is also called the \emph{domination principle}, and for $\nu=0$, the \emph{Frostman maximum principle}.

\begin{theorem}\label{dom}If\/ $U^\mu\leqslant U^\nu+q$ holds\/ $\mu$-a.e.\ {\rm(}$\mu$-almost everywhere\/{\rm)}, where\/ $\mu\in\mathcal E^+$, $\nu\in\mathfrak M^+$, and\/ $q\in[0,\infty)$, then this inequality holds on all of\/ $\mathbb R^n$.\end{theorem}

Let $\mathcal M\subset\mathcal E^+$ be a strongly closed convex cone containing $\mu=0$. Since then $\mathcal M$ is strongly complete, the following theorem is a particular case of \cite[Theorem~1.12.3, Proposition~1.12.4(2)]{E2}.

\begin{theorem}\label{th-proj}For any\/ $\mu\in\mathcal E^+$, there is a unique\/ $P\mu=P_{\mathcal M}(\mu)\in\mathcal M$ such that
\[\|\mu-P\mu\|=\inf_{\nu\in\mathcal M}\,\|\mu-\nu\|=:\varrho(\mu,\mathcal M).\]
This\/ $P\mu$ is called the orthogonal projection of\/ $\mu$ in the pre-Hilbert space\/ $\mathcal E$ onto\/ $\mathcal M$, and it is characterized uniquely by the relations
\begin{align}
\label{proj1}&(\mu-P\mu,\lambda)\leqslant0\text{ \ for all\ }\lambda\in\mathcal M,\\
\label{proj2}&(\mu-P\mu,P\mu)=0.
\end{align}
\end{theorem}

\section{Inner Riesz balayage}\label{sec-bal}

Unless explicitly stated otherwise, in what follows we assume that $A$ is an \emph{arbitrary\/} proper subset of $\mathbb R^n$ with $c_\alpha(A)>0$.

The notion of inner Riesz balayage of $\mu\in\mathfrak M^+$ to $A$ will be defined in three steps, presented respectively in Sections~\ref{st1}, \ref{st2}, and \ref{st3}.

\subsection{Step 1: $\mu\in\mathcal E^+$ and $A$ closed}\label{st1} Assume first that $\mu\in\mathfrak M^+$ has finite energy and $A$ is closed in $\mathbb R^n$.

\begin{theorem}\label{th-bala-f}
For\/ $\mu\in\mathcal E^+$ and\/ $A$ closed, there exists\/ $\mu^A\in\mathcal E^+_A$ such that
\begin{align}
\label{eq-bala-f1}U^{\mu^A}&=U^\mu\text{ \ n.e.\ on\ }A,\\
U^{\mu^A}&\leqslant U^\mu\text{ \ on\ }\mathbb R^n.\label{eq-bala-f2}
\end{align}
This\/ $\mu^A$ is actually the orthogonal projection of\/ $\mu$ in the pre-Hilbert space\/ $\mathcal E$ onto the convex cone\/ $\mathcal E^+_A$,
and it is determined uniquely within\/ $\mathcal E^+_A$ by {\rm(\ref{eq-bala-f1})}.
\end{theorem}

\begin{proof} For $A$ closed, the convex cone $\mathcal E^+_A$ is strongly closed, because the strong topology on $\mathcal E^+$ is stronger than the vague topology on $\mathcal E^+$ while $\mathfrak M^+_A$ is vaguely closed. According to Theorem~\ref{th-proj}, there exists therefore a unique orthogonal projection $P\mu=P_{\mathcal E^+_A}(\mu)$ of $\mu$ onto $\mathcal E^+_A$, and it is determined uniquely by (\ref{proj1}) and (\ref{proj2}) with $\mathcal M:=\mathcal E^+_A$. Relations (\ref{eq-bala-f1}) and (\ref{eq-bala-f2}) with $\mu^A:=P\mu$ can now be established in a manner similar to that in \cite[Proof of Theorem~4.16]{L}.

Indeed, the restriction $\lambda|_E$ of any $\lambda\in\mathcal E^+$ to the Borel set $E:=\bigl\{x\in A:\ U^\mu(x)>U^{P\mu}(x)\bigr\}$
belongs to $\mathcal E^+_A$, hence $(\mu-P\mu,\lambda|_E)\leqslant0$ by (\ref{proj1}), and consequently $\lambda|_E=0$. Since $\lambda\in\mathcal E^+$ has been chosen arbitrarily, $c_\alpha(E)=0$ according to Lemma~\ref{2.3.1}. Thus,
\begin{equation}\label{geq}U^{P\mu}\geqslant U^\mu\text{ \ n.e.\ on\ }A.\end{equation}

Another use of Lemma~\ref{2.3.1} now gives $U^{P\mu}\geqslant U^\mu$ $P\mu$-a.e., which together with (\ref{proj2}) shows that, actually,
$U^{P\mu}=U^\mu$ $P\mu$-a.e. By the domination principle (see Theorem~\ref{dom}), this yields (\ref{eq-bala-f2}), which combined with (\ref{geq}) establishes (\ref{eq-bala-f1}).

If (\ref{eq-bala-f1}) also holds with $\theta\in\mathcal E^+_A$, then $\nu:=\theta-P\mu\in\mathcal E$ and, furthermore, $U^\nu=0$ n.e.\ on $A$, hence $\nu$-a.e., again by Lemma~\ref{2.3.1}. We therefore obtain by integration $\|\nu\|=0$, which implies $\theta=P\mu$, the Riesz kernel being strictly positive definite.\end{proof}

\begin{remark}One could equally well write `\emph{q.e.}' (\emph{quasi everywhere\/}) instead of `n.e.' in (\ref{eq-bala-f1}), where `q.e.' refers to \emph{outer\/}  capacity \cite[Chapter~II, Section~2, n$^\circ$~6]{L}. Indeed, $\psi:=U^{\mu^A}-U^\mu$, being the difference between two l.s.c.\ functions, is Borel measurable, and hence $\{x\in A:\ \psi(x)\ne0\}$ is \emph{cap\-acit\-able\/} \cite[Theorem~2.8]{L}.\end{remark}

\begin{corollary}\label{rest-fin}If\/ $F$ is a closed subset of\/ $A$ with\/ $c_\alpha(F)>0$, then
\begin{equation}\label{eq-rest-fin}\mu^F=(\mu^A)^F\text{ \ for every\ }\mu\in\mathcal E^+.\end{equation}
\end{corollary}

\begin{proof}According to Theorem \ref{th-bala-f}, both $\mu^F$ and $(\mu^A)^F$ belong to $\mathcal E^+_F$ and, moreover,
\[U^{(\mu^A)^F}=U^{\mu^A}=U^\mu=U^{\mu^F}\text{ \ n.e.\ on\ }F.\]
Since (\ref{eq-bala-f1}) with $A=F$ determines $\mu^F$ uniquely within $\mathcal E^+_F$, (\ref{eq-rest-fin}) follows.\end{proof}

\subsection{Step 2: $\mu\in\mathcal E^+$ and $A$ arbitrary}\label{st2} Still requiring that $\mu\in\mathfrak M^+$ have finite energy, we now extend our analysis to $A$ arbitrary. Let  $\mathcal E_A'$ denote the strong cluster set of $\mathcal E^+_A$. Obviously, $\mathcal E_A'$ is a strongly closed convex cone in $\mathcal E^+$.

\begin{theorem}\label{th-f-cont}For\/ $\mu\in\mathcal E^+$ and\/ $A$ arbitrary, there is a unique\/ $\mu^A\in\mathcal E^+$ such that
\begin{equation}\label{eq-f-cont'}\mu^K\to\mu^A\text{ \ strongly and vaguely in\ }\mathcal E^+,\end{equation}
where\/ $K$ increases along the upper directed family\/ $\mathfrak C=\mathfrak C_A$ of all compact subsets of\/ $A$ and\/ $\mu^K$ is defined in Theorem\/~{\rm\ref{th-bala-f}}.
This\/ $\mu^A$ can alternatively be defined as the orthogonal projection of\/ $\mu$ onto\/ $\mathcal E_A'$, that is,\footnote{This implies that for $A$ closed, the measure $\mu^A$ determined by Theorem~\ref{th-f-cont} coincides with that determined by Theorem~\ref{th-bala-f}.}
\begin{equation}\label{proj'}\|\mu-\mu^A\|=\min_{\nu\in\mathcal E_A'}\,\|\mu-\nu\|=\varrho(\mu,\mathcal E_A').\end{equation}
\end{theorem}

\begin{proof}In view of (\ref{cap-incr}) and our assumption $c_\alpha(A)>0$, we may consider only those $K\in\mathfrak C$ whose capacity is ${}>0$. Since obviously
\[\|\mu-\mu^{K'}\|\leqslant\|\mu-\mu^K\|\text{ \ whenever\ }K\subset K',\]
Lemma~4.1.1 in \cite{F1} with $\mathcal H:=\mathcal E$, $\Gamma:=\{\mu-\nu:\ \nu\in\mathcal E^+_{K'}\}$, and $\lambda:=\mu-\mu^{K'}$ yields
\[\|\mu^K-\mu^{K'}\|^2=\|(\mu-\mu^K)-(\mu-\mu^{K'})\|^2\leqslant\|\mu-\mu^K\|^2-\|\mu-\mu^{K'}\|^2.\]
Being decreasing and lower bounded, the net $(\|\mu-\mu^K\|)_{K\in\mathfrak C}$ is Cauchy in $\mathbb R$, which together with the last display implies that the net $(\mu^K)_{K\in\mathfrak C}$ is strong Cauchy in $\mathcal E^+$. Being thus strongly bounded, $(\mu^K)_{K\in\mathfrak C}$ is vaguely bounded by \cite[Lemma~2.5.1]{F1}, and has a vague cluster point $\mu_0\in\mathfrak M^+$ according to \cite[Chapter~III, Section~2, Proposition~9]{B2}. Moreover, $\mu_0\in\mathcal E^+$ because the energy is vaguely l.s.c.\ on $\mathfrak M^+$ \cite[Eq.~(1.4.4)]{L}. Since the Riesz kernel is perfect (cf.\ Section~\ref{sec:prel}), $\mu^K\to\mu_0$ strongly in $\mathcal E^+$, and this $\mu_0$ is unique. As the vague topology on $\mathfrak M$ is Hausdorff, the unique vague cluster point $\mu_0$ of the net $(\mu^K)_{K\in\mathfrak C}$ has to be its vague limit \cite[Chapter~I, Section~9, n$^\circ$\,1]{B1}. This establishes (\ref{eq-f-cont'}) with $\mu^A:=\mu_0$.

It follows from (\ref{eq-f-cont'}) that $\mu^A\in\mathcal E_A'$ and, moreover,
\begin{equation}\label{stream}\varrho(\mu,\mathcal E_A^+)=\varrho(\mu,\mathcal E_A')\leqslant\|\mu-\mu^A\|=\lim_{K\uparrow A}\,\|\mu-\mu^K\|=\lim_{K\uparrow A}\,\varrho(\mu,\mathcal E_K^+),\end{equation}
the first equality being evident. On the other hand, for every $\nu\in\mathcal E^+_A$, $\nu|_K\to\nu$ vaguely as $K\uparrow A$ (see e.g.\ \cite[Lemma~1.2.2]{F1}), and therefore
\[\|\nu\|\leqslant\lim_{K\uparrow A}\,\|\nu|_K\|,\quad(\lambda,\nu)\leqslant\lim_{K\uparrow A}\,(\lambda,\nu|_K)\text{ \ for every\ }\lambda\in\mathcal E^+.\]
The opposite being obvious, equality in fact prevails in these inequalities; hence,
\[\|\mu-\nu\|=\lim_{K\uparrow A}\,\|\mu-\nu|_K\|\geqslant\lim_{K\uparrow A}\,\varrho(\mu,\mathcal E_K^+)\text{ \ for every\ }\nu\in\mathcal E^+_A,\]
and consequently
\[\varrho(\mu,\mathcal E^+_A)\geqslant\lim_{K\uparrow A}\,\varrho(\mu,\mathcal E_K^+).\]
Combining this with (\ref{stream}) establishes (\ref{proj'}).
\end{proof}

\begin{corollary}\label{cor-eq-arb}Both\/ {\rm(\ref{eq-bala-f1})} and\/ {\rm(\ref{eq-bala-f2})} remain valid for\/ $\mu^A$ defined in Theorem\/~{\rm\ref{th-f-cont}}. Furthermore, $U^{\mu^K}$ increases to\/ $U^{\mu^A}$ pointwise on\/ $\mathbb R^n$ as\/ $K\uparrow A$.
\end{corollary}

\begin{proof}For any $K,K'\in\mathfrak C_A$ such that $K\subset K'$, we see from (\ref{eq-bala-f2}) and (\ref{eq-rest-fin}) that
\[U^{\mu^K}=U^{(\mu^{K'})^K}\leqslant U^{\mu^{K'}}\leqslant U^\mu\text{ \ on\ }\mathbb R^n.\]
By Theorem~\ref{3.9}, there is therefore $\nu\in\mathfrak M^+$ such that
$U^{\mu^K}\uparrow U^\nu$ pointwise on $\mathbb R^n$ and $\mu^K\to\nu$ vaguely (as $K\uparrow A$). This together with (\ref{eq-f-cont'}) yields $\nu=\mu^A$, the vague topology on $\mathfrak M$ being Hausdorff, thereby establishing the latter assertion of the corollary.
Letting $K\uparrow A$ in the last display now gives (\ref{eq-bala-f2}).

Being the orthogonal projection of $\mu$ onto $\mathcal E_A'$, $\mu^A$ is characterized by (\ref{proj1}) and (\ref{proj2}) with $\mathcal M:=\mathcal E_A'$.
Having written (\ref{proj1}) for every $\nu\in\mathcal E'_A$, and then compared with (\ref{eq-bala-f2}), we obtain $(\mu-\mu^A,\nu)=0$, or equivalently
\begin{equation}\label{proj-proof}U^{\mu^A}=U^\mu\text{ \ $\nu$-a.e.\ for every\ }\nu\in\mathcal E'_A.\end{equation}
This leads to (\ref{eq-bala-f1}) in a way similar to that in the proof of Theorem~\ref{th-bala-f}. Indeed, for every $K\in\mathfrak C_A$ and every $\lambda\in\mathcal E^+$, consider $\lambda|_E$, where $E$ consists of all $x\in K$ with $U^{\mu^A}(x)<U^\mu(x)$. Since $\lambda|_E\in\mathcal E^+_A$, we have $\lambda|_E=0$ by (\ref{proj-proof}), and hence $c_\alpha(E)=0$ according to Lemma~\ref{2.3.1}.
\end{proof}

\begin{remark}\label{Borel}For $A$ Borel, Theorem~\ref{th-f-cont} and Corollary~\ref{cor-eq-arb} remain valid if $\mathfrak C_A$ is replaced by an increasing sequence $(A_k)_{k\in\mathbb N}$ of Borel sets, whose union equals $A$. This can be seen similarly as above, the only delicate point being in proving the vague convergence of $(\nu|_{A_k})_{k\in\mathbb N}$ to $\nu\in\mathfrak M^+_A$. This convergence is established by applying \cite[Chapter~IV, Section~1, Theorem~3]{B2} to $(1_{A_k}f)_{k\in\mathbb N}$, where $f\in C_0(\mathbb R^n)$ is positive and $1_Q$ denotes the indicator function of a set $Q$.\end{remark}

\begin{corollary}\label{th-f-contu}
For\/ $\mu\in\mathcal E^+$ and\/ $A$ arbitrary,
\begin{equation}\label{eq-f-sym}(\mu,\lambda^A)=(\mu^A,\lambda)\text{ \ for all\ }\lambda\in\mathcal E^+.\end{equation}
Identity\/ {\rm(\ref{eq-f-sym})} determines\/ $\mu^A$ uniquely; that is, if it holds for\/ $\nu\in\mathfrak M^+$ in place of\/ $\mu^A$, then\/ $\nu=\mu^A$.\end{corollary}

\begin{proof} Since $\lambda^A,\mu^A\in\mathcal E_A'$ for any $\lambda,\mu\in\mathcal E^+$, we obtain from (\ref{proj-proof})
\[(\mu-\mu^A,\lambda^A)=0\text{ \ and \ }(\lambda-\lambda^A,\mu^A)=0,\]
and (\ref{eq-f-sym}) follows by subtraction. If (\ref{eq-f-sym}) also holds for $\nu\in\mathfrak M^+$ in place of $\mu^A$, then
\[U^{\mu^A}\ast m^{(r)}=U^\nu\ast m^{(r)}\text{ \ for any\ }r>0,\]
where $m^{(r)}$ is the measure obtained by uniformly distributing unit mass over $B(0,r)$ and $\ast$ denotes convolution. Letting $r\to0$ and applying \cite[Theorem~1.11]{L} gives $\nu=\mu^A$ as claimed.\end{proof}

\begin{corollary}\label{place}For\/ $A$ arbitrary,
\[\mu^A=\mu\text{ \ for every\ }\mu\in\mathcal E'_A.\footnote{In particular, $\mu^A=\mu$ for every $\mu\in\mathcal E^+_A$. However, this is no longer valid (not even for $A$ closed) if we drop the requirement $E(\mu)<\infty$. For instance, $\mu^A=\mu$ does not hold for $\mu=\varepsilon_y$, provided that $y$ is an inner $\alpha$-ir\-reg\-ular point for $A$ (see Section~\ref{sec-wiener} below).}\]\end{corollary}

\begin{proof}This is obvious in view of (\ref{proj'}).\end{proof}

\subsection{Step 3: $\mu\in\mathfrak M^+$ and $A$ arbitrary}\label{st3} Assume for a moment that $\mu$ still has finite energy. In view of (\ref{eq-f-cont'}), we call the measure $\mu^A$ defined in Theorem~\ref{th-f-cont} the \emph{inner Riesz balayage of\/ $\mu$ to\/} $A$.  Since this $\mu^A$ is determined uniquely by the symmetry identity (\ref{eq-f-sym}), we are thus led to the following definition of inner Riesz balayage of $\mu\in\mathfrak M^+$ to $A$ arbitrary (cf.\ \cite[p.~257]{Ca2}).

\begin{definition}\label{def-bal}For $\mu\in\mathfrak M^+$ and $A$ arbitrary, we call $\mu^A\in\mathfrak M^+$ an \emph{inner balayage\/} of $\mu$ to $A$ if \begin{equation}\label{alternative}E(\mu,\lambda^A)=E(\mu^A,\lambda)\text{ \ for all\ }\lambda\in\mathcal E^+.\end{equation}
\end{definition}

\begin{theorem}\label{th-bala-2}For\/ $\mu\in\mathfrak M^+$ and\/ $A$ arbitrary, there exists a unique inner balayage\/ $\mu^A\in\mathfrak M^+$, and it satisfies
both\/ {\rm(\ref{eq-bala-f1})} and\/ {\rm(\ref{eq-bala-f2})}.
\end{theorem}

\begin{proof} Similarly as in \cite[p.~272]{L} (see also \cite[p.~257, footnote]{Ca2}), for $\mu\in\mathfrak M^+$ one can construct a sequence $(\mu_k)_{k\in\mathbb N}\subset\mathcal E^+$ such that
\begin{equation}\label{mon}U^{\mu_k}\uparrow U^\mu\text{ \ pointwise on\ }\mathbb R^n\text{ \ (as $k\to\infty$)}.\end{equation}
According to (\ref{proj-proof}) applied to each of those $\mu_k$,
\[U^{\mu_k^A}=U^{\mu_k}\leqslant U^{\mu_{k+1}}=U^{\mu_{k+1}^A}\text{ \ $\nu$-a.e.\ for every\ }\nu\in\mathcal E_A'.\]
In particular, $U^{\mu_k^A}\leqslant U^{\mu_{k+1}^A}$ $\mu_k^A$-a.e., which implies by the domination principle
\begin{equation}\label{appr-ineq}U^{\mu_k^A}\leqslant U^{\mu_{k+1}^A}\leqslant U^{\mu_{k+1}}\leqslant U^\mu\text{ \ on\ }\mathbb R^n,\end{equation}
the second inequality being valid by (\ref{eq-bala-f2}) (cf.\ Corollary~\ref{cor-eq-arb}).

Thus, $U^{\mu_k^A}$ increases along with $U^{\mu_k}$ (as $k\to\infty$) and does not exceed $U^\mu$. According to Theorem~\ref{3.9}, there exists $\nu\in\mathfrak M^+$ such that
$\mu_k^A\to\nu$ vaguely and
\begin{equation}\label{eq-appr}U^{\mu_k^A}\uparrow U^\nu\text{ \ pointwise on\ }\mathbb R^n\text{ \ (as $k\to\infty$)}.\end{equation}
Having now written (\ref{eq-f-sym}) for every $\mu_k\in\mathcal E^+$, $k\in\mathbb N$, and then applied
\cite[Chapter~IV, Section~1, Theorem~1]{B2}, which is possible in view of (\ref{mon}) and (\ref{eq-appr}), we get (\ref{alternative}) with $\mu^A:=\nu$.
The measure $\mu^A\in\mathfrak M^+$ is thus a required inner balayage of $\mu\in\mathfrak M^+$ to $A$, and its uniqueness follows from (\ref{alternative}) in the same manner as in Corollary~\ref{th-f-contu}.

Relation (\ref{eq-bala-f2}) is obtained directly from (\ref{appr-ineq}) and (\ref{eq-appr}). To prove (\ref{eq-bala-f1}), we observe from (\ref{mon}) and (\ref{eq-appr}) that $E:=\bigl\{x\in\mathbb R^n:\ U^{\mu^A}(x)<U^\mu(x)\bigr\}$ is contained in the union of the (Borel) sets $E_k:=\bigl\{x\in\mathbb R^n:\ U^{\mu_k^A}(x)<U^{\mu_k}(x)\bigr\}$, $k\in\mathbb N$. Since $c_\alpha(A\cap E_k)=0$ according to (\ref{eq-bala-f1}) applied to $\mu_k$ (cf.\ Corollary~\ref{cor-eq-arb}), Lemma~\ref{subad} shows that, indeed, $c_\alpha(A\cap E)=0$.
\end{proof}

\begin{remark}\label{alt-incr}For any sequence\/ {\rm(}net\/{\rm)} $(\mu_k)\subset\mathcal E^+$ such that (\ref{mon}) holds, we thus have
\[\mu_k^A\to\mu^A\text{ \ vaguely},\quad U^{\mu_k^A}\uparrow U^{\mu^A}\text{ \ pointwise on\ }\mathbb R^n,\] and either of these relations may be thought of as an alternative definition of $\mu^A$.\end{remark}

\begin{remark}We show in Theorem~\ref{incr} below that $\mu^K\to\mu^A$ vaguely as $K\uparrow A$, thereby justifying the term `\emph{inner\/} balayage'.\end{remark}

\begin{remark}\label{bal-support}In general, $\mu^A$ \emph{is not carried by\/} $A$, and this is the case even for the Newtonian kernel and an open ball.
What is clear so far is that $\mu^A$ is carried by $\overline{A}$. This will be specified in Theorem~\ref{desc-sup} below, providing a description of $S(\mu^A)$ for $A$ closed and $\mu\in\mathfrak M^+_{A^c}$. Also note that for any $a_1,a_2\in\mathbb R^1_+$ and any $\mu_1,\mu_2\in\mathfrak M^+$,
\begin{equation}\label{lin}(a_1\mu_1+a_2\mu_2)^A=a_1\mu_1^A+a_2\mu_2^A.\end{equation}
\end{remark}

\section{Further properties of inner balayage}\label{further} The following assertion shows that Definition~\ref{def-bal} can alternatively be given in terms of $\theta\in\mathfrak M^+$ in place of $\lambda\in\mathcal E^+$.

\begin{corollary}\label{th-alt2}For\/ $\mu\in\mathfrak M^+$ and\/ $A$ arbitrary,
\begin{equation}\label{alternative2}
E(\mu^A,\theta)=E(\mu,\theta^A)\text{ \ for all\ }\theta\in\mathfrak M^+.
\end{equation}
\end{corollary}

\begin{proof}For $\theta\in\mathfrak M^+$, choose $(\theta_k)_{k\in\mathbb N}\subset\mathcal E^+$ so that $U^{\theta_k}\uparrow U^\theta$ pointwise on $\mathbb R^n$. Then
$U^{\theta_k^A}\uparrow U^{\theta^A}$ pointwise on $\mathbb R^n$, cf.\ Remark~\ref{alt-incr}. But, according to (\ref{alternative}) with $\lambda=\theta_k$,
\[E(\mu^A,\theta_k)=E(\mu,\theta_k^A).\]
Letting here $k\to\infty$, we obtain (\ref{alternative2}) by \cite[Chapter~IV, Section~1, Theorem~1]{B2}.\end{proof}

\begin{corollary}\label{rest}For any\/ $\mu\in\mathfrak M^+$ and any\/ $Q\subset A$ with\/ $c_\alpha(Q)>0$,
\begin{equation}\label{ar-rest}\mu^Q=(\mu^A)^Q.\end{equation}
\end{corollary}

\begin{proof} For every $\lambda\in\mathcal E^+$, we obtain from (\ref{alternative})
\[E((\mu^A)^Q,\lambda)=E(\mu^A,\lambda^Q)=E(\mu,(\lambda^Q)^A)=E(\mu,\lambda^Q)=E(\mu^Q,\lambda),\]
the equality $(\lambda^Q)^A=\lambda^Q$ being valid by Corollary~\ref{place} applied to $\lambda^Q\in\mathcal E_Q'\subset\mathcal E_A'$.
Taking here $\lambda=m^{(r)}$ and letting $r\downarrow0$ gives (\ref{ar-rest}) (cf.\ the proof of Corollary~\ref{th-f-contu}).
\end{proof}

\begin{theorem}[{\rm Characteristic property}]\label{l-char} For\/ $\mu\in\mathfrak M^+$ and\/ $A$ arbitrary, $U^{\mu^A}$ can be characterized uniquely by the extremal property
\[U^{\mu^A}=\min_{\xi\in\Xi_A}\,U^\xi,\] where\/ $\Xi_A$ consists of all\/ $\xi\in\mathfrak M^+$ with
\begin{equation}\label{eq-char}U^\xi\geqslant U^\mu\text{ \ n.e.\ on\ }A.\end{equation}
\end{theorem}

\begin{proof}Since $\mu^A\in\Xi_A$ by (\ref{eq-bala-f1}) (cf.\ Theorem~\ref{th-bala-2}), it is enough to establish
\begin{equation}\label{eq-char1}U^{\mu^A}\leqslant U^\xi\text{ \ on\ }\mathbb R^n,\end{equation}
where $\xi\in\Xi_A$ is fixed. As $U^{\mu^A}$ is the pointwise limit of an increasing sequence $(U^{\mu_k^A})_{k\in\mathbb N}$ with $(\mu_k)_{k\in\mathbb N}\subset\mathcal E^+$ suitably chosen (cf.\ Remark~\ref{alt-incr}), it suffices to verify (\ref{eq-char1}) for $\mu\in\mathcal E^+$.
By (\ref{eq-bala-f2}) applied to $K\in\mathfrak C_A$, we have $U^{\mu^K}\leqslant U^\mu$ on $\mathbb R^n$, which together with (\ref{eq-char}) shows that the inequality
\[U^{\mu^K}\leqslant U^\xi\] holds n.e.\ on $K$, hence $\mu^K$-a.e.\ because $\mu^K\in\mathcal E^+_K$, and, consequently, on all of $\mathbb R^n$, by the domination principle. On account of Corollary~\ref{cor-eq-arb}, letting here $K\uparrow A$ results in (\ref{eq-char1}) as required.\end{proof}

\begin{corollary}\label{cor-red}For\/ $\mu\in\mathfrak M^+$ and\/ $A$ arbitrary,
\begin{equation}\label{redbal}\mu^A=\mu^{A'}\end{equation}
whenever\/ $(A\setminus A')\cup(A'\setminus A)$ is of inner capacity zero.
\end{corollary}

\begin{proof}Indeed, $U^{\mu^{A'}}=U^\mu$ n.e.\ on $A'$, hence n.e.\ on $A$. Therefore, by Theorem~\ref{l-char}, $U^{\mu^A}\leqslant
U^{\mu^{A'}}$ on $\mathbb R^n$. Since the same holds with $A$ and $A'$ reversed, (\ref{redbal}) follows.\end{proof}

\begin{theorem}\label{incr}For\/ $\mu\in\mathfrak M^+$ and\/ $A$ arbitrary, $U^{\mu^K}\uparrow U^{\mu^A}$ pointwise on\/ $\mathbb R^n$ and\/ $\mu^K\to\mu^A$ vaguely as\/ $K\uparrow A$. If, moreover, $E(\mu)<\infty$, then\/ $\mu^K\to\mu^A$ also strongly.\end{theorem}

\begin{proof} For $\mu\in\mathcal E^+$, this has already been established by Theorem~\ref{th-bala-f} and Corollary~\ref{cor-eq-arb}. It thus remains to prove the former assertion  for $\mu\in\mathfrak M^+\setminus\mathcal E^+$.

According to (\ref{ar-rest}), for any $K,K'\in\mathfrak C_A$ such that $K\subset K'$, we have $\mu^K=(\mu^{K'})^K$. In view of (\ref{eq-bala-f2}), this implies that the net $(U^{\mu^K})_{K\in\mathfrak C_A}$ is increasing and majorized by $U^\mu$. By Theorem~\ref{3.9}, there is therefore $\nu\in\mathfrak M^+$ such that
$U^{\mu^K}\uparrow U^\nu$ pointwise on $\mathbb R^n$ and $\mu^K\to\nu$ vaguely as $K\uparrow A$.
The proof is completed by showing that $\nu=\mu^A$, or equivalently
\[E(\nu,\lambda)=E(\mu,\lambda^A)\]
for any given $\lambda\in\mathcal E^+$ (cf.\ Definition~\ref{def-bal}). Indeed, according to (\ref{alternative}) applied to $K$,
\[E(\mu^K,\lambda)=E(\mu,\lambda^K)\text{ \ for every\ }K\in\mathfrak C_A,\]
while $U^{\lambda^A}$ is the pointwise limit of the increasing net $(U^{\lambda^K})_{K\in\mathfrak C_A}$ (see Corollary~\ref{cor-eq-arb}).
Letting $K\uparrow A$, we obtain the required identity by \cite[Chapter~IV, Section~1, Theorem~1]{B2}.\end{proof}

\begin{corollary}\label{cor-decr}Fix\/ $\mu\in\mathcal E^+$ and\/ $A\subset\mathbb R^n$. For every\/ $\varepsilon>0$, there exists\/ $K_0\in\mathfrak C_A$ with the property
$\|\mu^A-\mu^K\|<\varepsilon$ for all\/ $K\in\mathfrak C_A$ such that\/ $K\supset K_0$.
\end{corollary}

\begin{theorem}\label{incr'}For\/ $A$ Borel, Theorem\/~{\rm\ref{incr}} and Corollary\/~{\rm\ref{cor-decr}} remain valid if\/ $\mathfrak C_A$ is replaced by an increasing sequence\/ $(A_k)_{k\in\mathbb N}$ of Borel sets, whose union equals\/ $A$.\end{theorem}

\begin{proof}Since for every $\lambda\in\mathcal E^+$, $U^{\lambda^{A_k}}\uparrow U^{\lambda^A}$ pointwise on $\mathbb R^n$ as $k\to\infty$ (see Remark~\ref{Borel}), the proof runs as above.\end{proof}

The following result can certainly be extended to a general perfect kernel on a locally compact space, which is, however, outside the scope of this study.

\begin{theorem}\label{th-decr}Assume\/ $A$ is the intersection of a lower directed family\/ $(A_t)_{t\in T}$ of closed sets. For any\/ $\mu\in\mathcal E^+$,
$\mu^{A_t}\to\mu^A$ strongly and vaguely.
\end{theorem}

\begin{proof}For $\mu\in\mathcal E^+$, $\mu^{A_t}$ is the orthogonal projection of $\mu$ onto $\mathcal E_{A_t}^+$ (cf.\ Theorem~\ref{th-bala-f}). Since
$\mathcal E_A^+\subset\mathcal E_{A_{t'}}^+\subset\mathcal E_{A_t}^+$ whenever $t'$ follows $t$, the net $\bigl(\varrho(\mu,\mathcal E_{A_t}^+)\bigr)_{t\in T}$ is increasing and majorized by $\varrho(\mu,\mathcal E_A^+)<\infty$; hence, it is Cauchy in $\mathbb R$. In consequence of \cite[Lemma~4.1.1]{F1} with $\mathcal H:=\mathcal E$, $\Gamma:=\{\mu-\nu:\ \nu\in\mathcal E_{A_t}^+\}$, and $\lambda:=\mu-\mu^{A_t}$,
\[\|\mu^{A_t}-\mu^{A_{t'}}\|^2=\|(\mu-\mu^{A_t})-(\mu-\mu^{A_{t'}})\|^2\leqslant\varrho(\mu,\mathcal E_{A_{t'}}^+)^2-\varrho(\mu,\mathcal E_{A_t}^+)^2.\]
It follows that the net $(\mu^{A_t})_{t\in T}$ is strong Cauchy in $\mathcal E^+$. Being thus strongly bounded, it is vaguely bounded by \cite[Lemma~2.5.1]{F1} and, therefore, has a vague cluster point $\mu_0\in\mathfrak M^+$ according to \cite[Chapter~III, Section~2, Proposition~9]{B2}. On account of the perfectness of the Riesz kernel, this $\mu_0$ is, in fact, a (unique) strong and vague limit of $(\mu^{A_t})_{t\in T}$. Consequently, $\mu_0$ belongs to $\mathcal E^+_{A_t}$ for every $t\in T$, and hence to $\mathcal E^+_A$, the intersection of $\mathcal E^+_{A_t}$ over $t\in T$. Thus,
\[\varrho(\mu,\mathcal E^+_A)\leqslant\|\mu-\mu_0\|=\lim_{t}\,\|\mu-\mu^{A_t}\|=\lim_{t}\,\varrho(\mu,\mathcal E^+_{A_t})\leqslant\varrho(\mu,\mathcal E^+_A),\]
which implies $\mu_0=\mu^A$.\end{proof}

The following corollary to Theorem~\ref{incr} will be specified in Theorems~\ref{bal-mass-th} and \ref{bal-mass-th'} below.

\begin{corollary}\label{ineq-en}For\/ $\mu\in\mathfrak M^+$ and\/ $A$ arbitrary,
\begin{equation}\label{t-mass}\mu^A(\mathbb R^n)\leqslant\mu(\mathbb R^n).\end{equation}
\end{corollary}

\begin{proof} Since $\mu^K\to\mu^A$ vaguely as $K\uparrow A$, while the map $\nu\mapsto\nu(\mathbb R^n)$ is vaguely l.s.c.\ on $\mathfrak M^+$, it suffices to establish (\ref{t-mass}) for $A=K$ compact. Consider a closed ball $\overline{B}$ containing $K$, and the equilibrium measure $\gamma$ on $\overline{B}$; then $U^\gamma=1$ on $\overline{B}$ and $U^\gamma\leqslant1$ on $\mathbb R^n$ \cite[Chapter~II, Section~3, n$^\circ$~13]{L}. Therefore,
\[\mu^K(\mathbb R^n)=\int U^\gamma\,d\mu^K=\int U^{\mu^K}\,d\gamma\leqslant
\int U^\mu\,d\gamma=\int U^\gamma\,d\mu\leqslant\mu(\mathbb R^n),\]
the former inequality being valid according to (\ref{eq-bala-f2}) (cf.\ Theorem~\ref{th-bala-2}).
\end{proof}

\section{Inner Riesz equilibrium measure. Criteria for its existence}

We assume as above that $A$ is an arbitrary proper subset of $\mathbb R^n$ with $c_\alpha(A)>0$, and denote by $\Theta_A$ the class of all $\nu\in\mathfrak M^+$ with  $U^\nu\geqslant1$ n.e.\ on $A$.

\begin{definition}\label{def-eq}$\gamma_A\in\mathfrak M^+$ is said to be an \emph{inner equilibrium measure\/} of $A$ if (its potential $U^{\gamma_A}$ is not identically infinite and)\footnote{One can introduce a concept of inner balayage for positive $\alpha$-sup\-er\-harmonic functions on $\mathbb R^n$, generalizing that by Cartan for $\alpha=2$ \cite[p.~257]{Ca2}, and then $U^{\gamma_A}$ (if $\gamma_A$ exists) will be thought of as an inner balayage of $f\equiv1$ to $A$. Being however mainly concerned with the existence of $\gamma_A$, we drop this part of the analysis.}
\[U^{\gamma_A}=\inf_{\nu\in\Theta_A}\,U^\nu.\]
\end{definition}

An inner equilibrium measure $\gamma_A$ is certainly \emph{unique}, and it exists only if $\Theta_A$ is nonempty. (We shall show in Lemma~\ref{l-eq-ar} below that the latter can actually be reversed, and so $\Theta_A\ne\varnothing$ is necessary and sufficient for the existence of $\gamma_A$.)

\begin{lemma}\label{l-eq-finite}Assume that\/ $c_\alpha(A)<\infty$. Then\/ $\gamma_A$ exists and, moreover,
\begin{itemize}
\item[{\rm(a)}] $\gamma_A(\mathbb R^n)=E(\gamma_A)=c_\alpha(A)$,
\item[{\rm(b)}] $S(\gamma_A)\subset\overline{A}$,
\item[{\rm(c)}] $U^{\gamma_A}=1$ n.e.\ on $A$,
\item[{\rm(d)}] $U^{\gamma_A}\leqslant1$ on $\mathbb R^n$.
\end{itemize}
This\/ $\gamma_A$ is a unique solution to the problem of minimizing the energy over\/ $\Theta_A\cap\mathcal E^+$, and hence it is characterized uniquely within\/ $\mathcal E^+$ by\/ {\rm(a)} and\/ {\rm(c)}.
\end{lemma}

\begin{proof}This is obtained from \cite[Chapter~II, Section~2, n$^\circ$~7]{L} and \cite[Lemma~4.5]{L}. See also \cite[Section~4.1]{F1}.\end{proof}

\begin{lemma}\label{l-eq-ar}
For\/ $A$ arbitrary, assume that\/ $\Theta_A\ne\varnothing$. Then\/ $\gamma_A$ exists. Furthermore, it is absolutely continuous and has the properties\/ $S(\gamma_A)\subset\overline{A}$ and\/\footnote{Relation (\ref{eq-ex0}) will be specified below (cf.\ Lemma~\ref{reg-com} and Theorem~\ref{irreg-com}). See also Theorem~\ref{desc-eq}, which establishes a detailed description of $U^{\gamma_A}$ and $S(\gamma_A)$ for $A$ closed.}
\begin{align}
\label{eq-ex0}&U^{\gamma_A}=1\text{ \ n.e.\ on\ }A,\\
\label{eq-ex1}&U^{\gamma_A}\leqslant1\text{ \ on\ }\mathbb R^n,\\
\label{eq-ex2}&\gamma_K\to\gamma_A\text{ \ vaguely as\ }K\uparrow A,\\
\label{eq-ex3}&U^{\gamma_K}\uparrow U^{\gamma_A}\text{ \ pointwise on $\mathbb R^n$ as\ }K\uparrow A.\end{align}
\end{lemma}

\begin{proof}Fix $\xi\in\mathfrak M^+$ with $U^\xi\geqslant1$ n.e.\ on $A$. Then for any $K,K'\in\mathfrak C_A$ such that $K\subset K'$, we have $U^{\gamma_K}=U^{\gamma_{K'}}\leqslant U^\xi$ n.e.\ on $K$ (cf.\ Lemma~\ref{l-eq-finite}(c)), and hence $\gamma_K$-a.e. The domination principle therefore shows that the net $(U^{\gamma_K})_{K\in\mathfrak C_A}$ is pointwise increasing on $\mathbb R^n$ and majorized by $U^\xi$. According to Theorem~\ref{3.9}, there exists $\nu_0\in\mathfrak M^+$ such that
$U^{\gamma_K}\uparrow U^{\nu_0}$ pointwise on $\mathbb R^n$ and $\gamma_K\to\nu_0$ vaguely as $K\uparrow A$.
Hence, $U^{\nu_0}\leqslant1$ on $\mathbb R^n$ because $U^{\gamma_K}\leqslant1$ on $\mathbb R^n$ by Lemma~\ref{l-eq-finite}(d).

We claim that $U^{\nu_0}=1$ n.e.\ on $A$, or equivalently n.e.\ on every $K\in\mathfrak C_A$. We thus need to prove $c_\alpha(E)=0$, where
$E:=\{x\in K:\ U^{\nu_0}(x)<1\}$. But this is obvious in view of the relations $E\subset E':=\{x\in K:\ U^{\gamma_K}(x)<1\}$ and $c_\alpha(E')=0$, the latter being seen from Lemma~\ref{l-eq-finite}(c).

Thus, $\nu_0\in\Theta_A$. We assert that this $\nu_0$ actually serves as an inner equilibrium measure of $A$. According to Definition~\ref{def-eq}, it is enough to verify that
$U^{\nu_0}\leqslant U^\nu$ on $\mathbb R^n$ for every $\nu\in\mathfrak M^+$ with $U^\nu\geqslant1$ n.e.\ on $A$. Since then $U^{\gamma_K}\leqslant U^\nu$ on $\mathbb R^n$ for every $K\in\mathfrak C_A$ (see above), the required inequality is obtained by letting $K\uparrow A$.

It has thus been proven that $\gamma_A:=\nu_0$ exists and satisfies (\ref{eq-ex0})--(\ref{eq-ex3}). Next, it follows from (\ref{eq-ex2}) that $S(\gamma_A)\subset\overline{A}$, $\mathfrak M^+(\overline{A})$ being vaguely closed. Finally, since the restriction of $\gamma_A$ to any compact subset of $\mathbb R^n$ is of finite energy because of (\ref{eq-ex1}), $\gamma_A$ is absolutely continuous.\end{proof}

\begin{corollary}\label{clos-eq}For\/ $A$ closed, assume that there is an absolutely continuous measure\/ $\check{\gamma}\in\mathfrak M^+_A$ with\/ $U^{\check{\gamma}}=1$ n.e.\ on\/ $A$. Then\/ $\gamma_A$ exists and, moreover, $\gamma_A=\check{\gamma}$.\end{corollary}

\begin{proof} Since $\check{\gamma}\in\Theta_A$, $\gamma_A$ exists according to the preceding lemma. Furthermore, $\gamma_A$ is absolutely continuous, supported by $A$, and satisfies (\ref{eq-ex0}). Then necessarily $\gamma_A=\check{\gamma}$, because any two absolutely continuous measures of the class $\mathfrak M^+_A$ coincide whenever their potentials are equal n.e.\ on $A$ (cf.\ \cite[p.~178, Remark]{L}).\end{proof}

In a particular case where $A$ is Borel, Theorem~\ref{l-eq-ex} below is \cite[Theorem~5.1]{L}. However, \cite[Theorem~5.1]{L} has not been completely justified, because the proof of its necessity part is based on the concept of balayage, introduced in \cite[Chapter~IV, Section~6, n$^\circ$~25]{L} by means of the integral representation (\ref{L-repr}) (see the Introduction for details). Applying the theory of inner Riesz balayage, developed in Section~\ref{sec-bal} above, we fix the gap in \cite[Proof of Theorem~5.1]{L} and, moreover, we strengthen \cite[Theorem~5.1]{L} to $A$ arbitrary.

\begin{theorem}\label{l-eq-ex}For\/ $A$ arbitrary, the following two assertions are equivalent.
\begin{itemize}
\item[{\rm(i)}]There exists an inner equilibrium measure\/ $\gamma_A$.
\item[{\rm(ii)}]If\/ $R_k:=\bigl\{x\in\mathbb R^n: q^k\leqslant|x|<q^{k+1}\bigr\}$, where\/ $q\in(1,\infty)$, then
\begin{equation}\label{iii}\sum_{k\in\mathbb N}\,\frac{c_\alpha(A\cap R_k)}{q^{k(n-\alpha)}}<\infty.\end{equation}
\end{itemize}
\end{theorem}

\begin{proof}Assuming first that (ii) holds, write $\gamma_k:=\gamma_{A_k}$, where $A_k:=A\cap R_k$; this $\gamma_{A_k}$ exists according to Lemma~\ref{l-eq-finite}. It is seen from Lemma~\ref{l-eq-ar} that (i) will follow once we establish
\begin{equation}\label{xi}\xi:=\sum_{k\in\mathbb N}\,\gamma_k\in\Theta_A.\end{equation}
To this end, we first observe that $U^\xi\not\equiv\infty$, or equivalently (cf.\ (\ref{1.3.10}))
\[I:=\int_{|x|\geqslant q}\frac{d\,\xi(x)}{|x|^{n-\alpha}}<\infty.\]
Indeed,
\[I=\sum_{k\in\mathbb N}\,\int_{R_k}\frac{d\,\xi(x)}{|x|^{n-\alpha}}\leqslant\sum_{k\in\mathbb N}\,\int\frac{d\,\gamma_k(x)}{|x|^{n-\alpha}}\leqslant
\sum_{k\in\mathbb N}\,\frac{\gamma_k(\mathbb R^n)}{q^{k(n-\alpha)}}=\sum_{k\in\mathbb N}\,\frac{c_\alpha(A_k)}{q^{k(n-\alpha)}},\]
the last equality being valid by Lemma~\ref{l-eq-finite}(a); therefore, $I<\infty$ by (\ref{iii}). The proof of (\ref{xi}) is thus reduced to establishing $U^\xi\geqslant 1$ n.e.\ on $A$. Since the sets $R_k$, $k\in\mathbb N$, are Borel, this follows from $U^\xi\geqslant U^{\gamma_k}=1$ n.e.\ on $A\cap R_k$ by applying Lemma~\ref{subad}.

Assuming now that $\gamma_A$ exists, we complete the proof by showing
\begin{equation*}\label{ser}S_1:=\sum_{k\in\mathbb N}\,\frac{c_\alpha(A_{2k})}{q^{2k(n-\alpha)}}<\infty,\quad
S_2:=\sum_{k\in\mathbb N}\,\frac{c_\alpha(A_{2k-1})}{q^{(2k-1)(n-\alpha)}}<\infty\end{equation*}
(cf.\ (\ref{iii})). Since either of these series can be handled in the same manner as the other, we shall establish $S_1<\infty$. Write
\[A':=\bigcup_{k\in\mathbb N}\,A_{2k},\]
$\gamma':=\gamma_A|_{\overline{A'}}$, $\gamma'':=\gamma_A-\gamma'$, and
\[\tilde{\gamma}:=\gamma'+(\gamma'')^{A'},\] where the existence of the inner balayage $(\gamma'')^{A'}$ (cf.\ Definition~\ref{def-bal}) is justified by Theorem~\ref{th-bala-2} (compare with \cite[Proof of Theorem~5.1]{L}).  According to (\ref{eq-bala-f1}) (cf.\ Theorem~\ref{th-bala-2}), $U^{(\gamma'')^{A'}}=U^{\gamma''}$ n.e.\ on $A'$, and therefore
\begin{equation}\label{5.1}U^{\tilde{\gamma}}=U^{\gamma'+(\gamma'')^{A'}}=U^{\gamma'+\gamma''}=U^{\gamma_A}=1\text{ \ n.e.\ on\ }A'.\end{equation}

Noting that
\[\sum_{k\in\mathbb N}\,\frac{\tilde{\gamma}\bigl(\,\overline{A_{2k}}\,\bigr)}{q^{(2k+1)(n-\alpha)}}\leqslant\sum_{k\in\mathbb N}\,\int_{\overline{R_{2k}}}\frac{d\,\tilde{\gamma}(x)}{|x|^{n-\alpha}}\leqslant
\int_{|x|>1}\frac{d\,\tilde{\gamma}(x)}{|x|^{n-\alpha}}<\infty,\]
the last inequality being valid by (\ref{1.3.10}), we obtain
\[\sum_{k\in\mathbb N}\,\frac{\tilde{\gamma}\bigl(\,\overline{A_{2k}}\,\bigr)}{q^{2k(n-\alpha)}}<\infty.\]
Hence, $S_1<\infty$ will follow if we show
\[c_\alpha(A_{2k})\leqslant M\tilde{\gamma}\bigl(\,\overline{A_{2k}}\,\bigr)\text{ \ for all\ }k\in\mathbb N,\]
where $M\in(0,\infty)$ is independent of $k$. The proof of this is based on (\ref{5.1}) and runs in the same manner as in \cite[pp.~282--283]{L}.
\end{proof}

\begin{remark}\label{fin-cap} The finiteness of $c_\alpha(A)$ is sufficient for the existence of an inner equilibrium measure $\gamma_A$ (cf.\ Lemma~\ref{l-eq-finite}), but not necessary. This can be seen by comparing Theorem~\ref{l-eq-ex} with the following assertion.

\begin{lemma}[{\rm see \cite[Lemma~5.5]{L}}]\label{f-en} For\/ $A$ Borel, $c_\alpha(A)<\infty$ holds if and only if
\[\sum_{k\in\mathbb N}\,\frac{c_\alpha(A\cap R_k)}{q^{2k(n-\alpha)}}<\infty,\]
where\/ $R_k$, $k\in\mathbb N$, are defined in Theorem\/~{\rm\ref{l-eq-ex}} and\/ $q\in(1,\infty)$.
\end{lemma}
\end{remark}

\begin{figure}[htbp]
\begin{center}
\vspace{-.4in}
\hspace{1in}\includegraphics[width=4in]{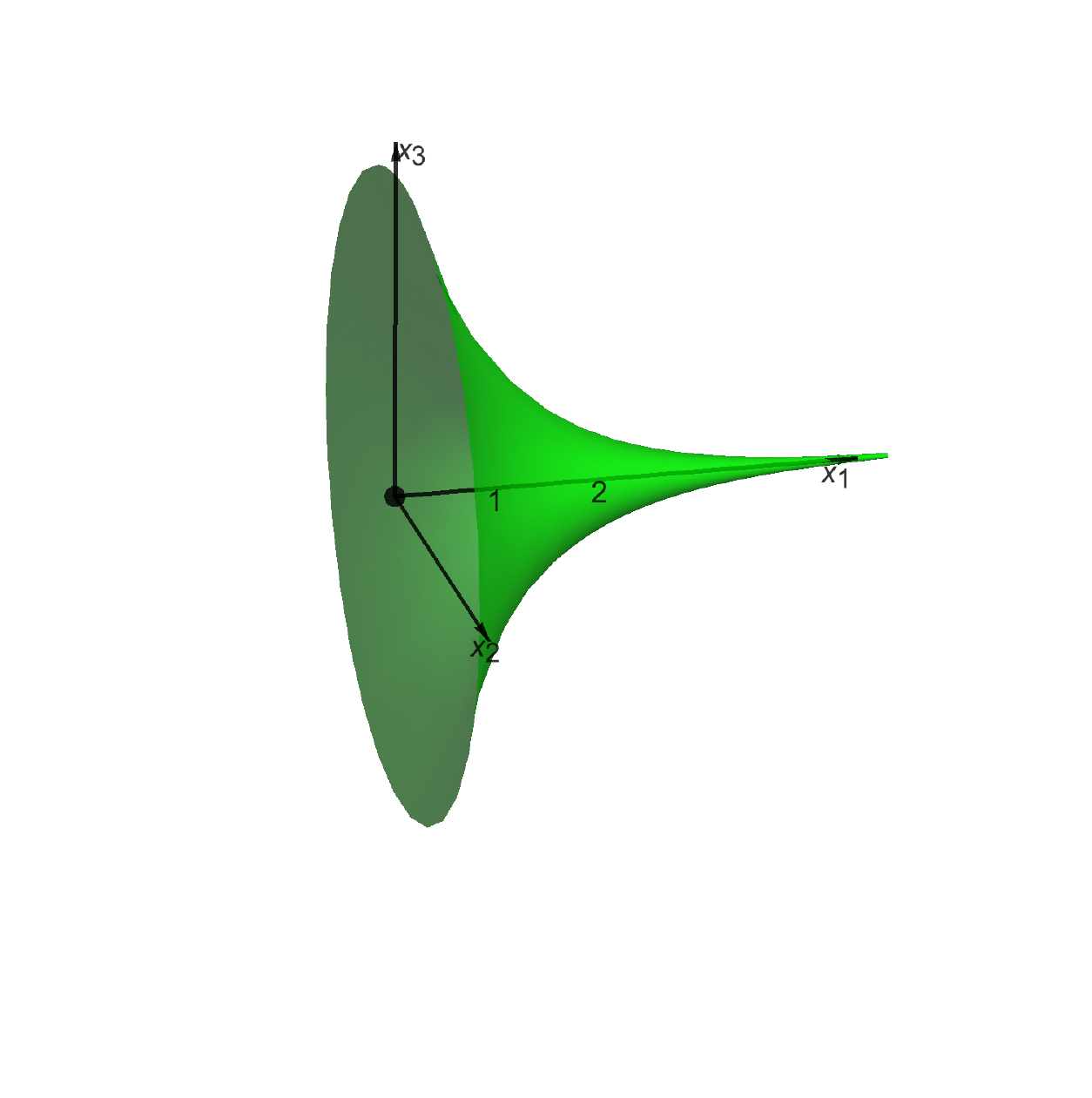}
\vspace{-.8in}
\caption{$A:=\bigl\{0\leqslant x_1<\infty,  \ x_2^2+x_3^2\leqslant\rho^2(x_1)\bigr\}$
with $\rho(x_1)=\exp(-x_1)$.\vspace{-.1in}}
\label{Fig1}
\end{center}
\end{figure}

The next example can be obtained from Theorem~\ref{l-eq-ex} and Lemma~\ref{f-en} by analyzing estimates in \cite[Chapter~V, Section~1, Example]{L}.

\begin{example}\label{ex-thin}Let $n=3$ and $\alpha=2$. Define $A$ to be a rotation body
\begin{equation}\label{defA}A:=\bigl\{x\in\mathbb R^3: \ 0\leqslant x_1<\infty, \ x_2^2+x_3^2\leqslant\varrho^2(x_1)\bigr\},\end{equation}
where $\varrho$ is given by one of the following three formulae:
\begin{align}
\label{c1}\varrho(x_1)&=x_1^{-s}\text{ \ with \ }s\in[0,\infty),\\
\label{c2}\varrho(x_1)&=\exp(-x_1^s)\text{ \ with \ }s\in(0,1],\\
\label{c3}\varrho(x_1)&=\exp(-x_1^s)\text{ \ with \ }s\in(1,\infty).
\end{align}
Then $\gamma_A$ does not exist if $\varrho$ is defined by (\ref{c1}), $\gamma_A$ exists but $c_2(A)=\infty$ if $\varrho$ is given by (\ref{c2})
(see Figure~\ref{Fig1}),
and finally $c_2(A)<\infty$ if (\ref{c3}) holds.
\end{example}

\section{Wiener type criterion of inner $\alpha$-irregularity}\label{sec-wiener}

\begin{definition}\label{def-reg}A point $y\in\mathbb R^n$ is said to be \emph{inner\/ $\alpha$-irregular\/} for $A$ if $y\in\overline{A}$ and $\varepsilon_y^A\ne\varepsilon_y$, where $\varepsilon_y^A$ is the inner balayage of $\varepsilon_y$ to $A$ (cf.\ Definition~\ref{def-bal}). All other points of $\overline{A}$ are said to be \emph{inner\/ $\alpha$-regular\/} for $A$.\end{definition}

\begin{remark}\label{en-fin}For every $y\notin\overline{A}$, $\varepsilon_y^A\in\mathcal E^+$, and therefore $\varepsilon_y^A\ne\varepsilon_y$. Indeed, in view of (\ref{eq-bala-f2}) (cf.\ Theorem~\ref{th-bala-2}),
\[U^{\varepsilon_y^A}(x)\leqslant U^{\varepsilon_y}(x)=|x-y|^{\alpha-n}\leqslant\max_{z\in\overline{A}}\,|z-y|^{\alpha-n}<\infty\text{ \ for all\ }x\in\overline{A}.\]
Since $\varepsilon_y^A(\mathbb R^n)\leqslant1$ according to (\ref{t-mass}), $E(\varepsilon_y^A)<\infty$ follows.\end{remark}

\begin{lemma}\label{l-regg} $y$ is inner\/ $\alpha$-regular for\/ $A$ {\rm(}if and\/{\rm)} only if
\begin{equation}\label{eq-reg1}U^{\mu^A}(y)=U^\mu(y)\text{ \ for every\ }\mu\in\mathfrak M^+.\end{equation}
\end{lemma}

\begin{proof}If $y$ is inner $\alpha$-regular for $A$, then (\ref{alternative2}) applied to $\theta=\varepsilon_y$ gives
\[U^{\mu^A}(y)=E(\mu^A,\varepsilon_y)=E(\mu,\varepsilon_y^A)=E(\mu,\varepsilon_y)=U^\mu(y)\]
for every $\mu\in\mathfrak M^+$, which is our claim.\end{proof}

Let $A_I$ consist of all inner $\alpha$-ir\-reg\-ular points for $A$.

\begin{theorem}[{\rm Wiener type criterion}]\label{Wiener}$y\in A_I$ if and only if
\begin{equation}\label{w}\sum_{k\in\mathbb N}\,\frac{c_\alpha(A_k)}{q^{k(n-\alpha)}}<\infty,\end{equation}
where\/ $A_k:=A\cap\bigl\{x\in\mathbb R^n: q^{k+1}<|x-y|\leqslant q^k\bigr\}$ and\/ $q\in(0,1)$.
\end{theorem}

Theorem~\ref{Wiener} follows directly from Lemmas~\ref{l1}--\ref{l3} below. Theorem~\ref{Wiener} implies, in turn, the next two assertions (see  Section~\ref{proof} for the proof of the latter).

\begin{corollary}\label{reg-inner}$A_I\subset\partial A$.\end{corollary}

\begin{theorem}\label{irreg-com}$c_\alpha(A_I\cap A)=0$.\end{theorem}

\subsection{Auxiliary assertions}\label{aux} For any $y\in\mathbb R^n$, define the inversion $J_y$ with respect to $S(y,1)$ mapping each point $x\ne y$ to the point $x^*$ on the ray through $x$ issuing from $y$ which is determined uniquely by
\[|x-y|\cdot|x^*-y|=1.\]
This is a homeomorphism of $\mathbb R^n\setminus\{y\}$ onto itself; furthermore,
\begin{equation}\label{inv}|x^*-z^*|=\frac{|x-z|}{|x-y||z-y|}\text{ \ for all\ }x,z\in\mathbb R^n\setminus\{y\}.\end{equation}
It can be extended to a homeomorphism of the one-point compactification $\overline{\mathbb R^n}$ of $\mathbb R^n$ onto itself such that $y$ and the point at infinity are mapped to each other.

In Lemmas~\ref{l1}--\ref{l3} below, $y\in\mathbb R^n$ is fixed and $A^*$ the $J_y$-im\-age of $A\setminus\{y\}$.

\begin{lemma}\label{l1} Relation\/ {\rm(\ref{w})} holds if and only if an inner equilibrium measure\/ $\gamma_{A^*}$ of\/ $A^*$ exists.\end{lemma}

\begin{proof} Let $q\in(0,1)$ and $A_k$ be as in Theorem~\ref{Wiener}. It follows from (\ref{inv}) that
\[q^{-2k}|x-z|\leqslant|x^*-z^*|\leqslant q^{-2k-2}|x-z|\text{ \ for any\ }x,z\in A_k,\]
and hence, by \cite[Remark to Theorem~2.9]{L},
\begin{equation}\label{est}q^{-2k(n-\alpha)}c_\alpha(A_k)\leqslant c_\alpha(A_k^*)\leqslant q^{-(2k+2)(n-\alpha)}c_\alpha(A_k),\end{equation}
where \[A_k^*:=J_y(A_k)=A^*\cap\bigl\{x\in\mathbb R^n:\ q^{-k}\leqslant|x-y|<q^{-k-1}\bigr\}.\]
Therefore, (\ref{w}) holds if and only if
\[\sum_{k\in\mathbb N}\,q^{k(n-\alpha)}c_\alpha(A_k^*)<\infty,\]
which according to Theorem~\ref{l-eq-ex} is equivalent to the existence of $\gamma_{A^*}$.\end{proof}

To each $\nu\in\mathfrak M^+$ with $\nu(\{y\})=0$ we assign the \emph{Kelvin transform\/}
$\nu^*=\mathcal K_y\nu=\mathcal K_y(\nu)\in\mathfrak M^+$ (see \cite{R} or \cite[Chapter IV, Section 5, n$^\circ$~19]{L}) by means of
\begin{equation}\label{kelv-m}d\nu^*(x^*)=|x-y|^{\alpha-n}\,d\nu(x),\text{ \ where\ }x^*:=J_y(x)\in\mathbb R^n.\end{equation}
Then $\mathcal K_y$ is an involution, i.e.
$\mathcal K_y(\mathcal K_y\nu)=\nu$, which implies in view of (\ref{kelv-m})
\begin{equation}\label{kelv-mmm}\nu(\mathbb R^n)=U^{\nu^*}(y).\end{equation}
Next, combining (\ref{kelv-m}) and (\ref{inv}) yields
\begin{equation}\label{KP}U^{\nu^*}(x^*)=|x-y|^{n-\alpha}U^\nu(x)\text{ \ for all\ }x^*\in\mathbb R^n,\end{equation}
and therefore
\begin{equation}\label{K}E(\mu^*,\nu^*)=E(\mu,\nu)\end{equation}
for every $\mu\in\mathfrak M^+$ with $\mu(\{y\})=0$. Equality (\ref{K}) is obtained by multiplying (\ref{kelv-m}) (with $\mu$ in place of $\nu$) by (\ref{KP}), and then integrating with respect to $d\mu(x)$ over $\mathbb R^n$.

\begin{lemma}\label{l2} Assume\/ $\gamma_{A^*}$ exists. For its Kelvin transform\/ $(\gamma_{A^*})^*=\mathcal K_y\gamma_{A^*}$, we have
\begin{equation}\label{baleq}(\gamma_{A^*})^*=\varepsilon_y^A,\end{equation}
and hence\/ $\varepsilon_y^A$ is absolutely continuous along with\/ $\gamma_{A^*}$.\end{lemma}

\begin{proof}Assume that $A\not\ni y$, which certainly involves no loss of generality (cf.\ Corollary~\ref{cor-red}). Then $J_y$ is an or\-der-pre\-serv\-ing one-to-one mapping of  $\mathfrak C_A$ onto $\mathfrak C_{A^*}$. In view of the absolute continuity of inner equilibrium measure (cf.\ Lemma~\ref{l-eq-ar}), one can consider the Kelvin transforms $(\gamma_{A^*})^*=\mathcal K_y\gamma_{A^*}$ and $(\gamma_{K^*})^*=\mathcal K_y\gamma_{K^*}$ for every $K^*:=J_y(K)\in\mathfrak C_{A^*}$. It follows from (\ref{eq-ex3}), applied to $A^*$, and (\ref{KP}), applied to each of $\gamma_{K^*}$ and $\gamma_{A^*}$, that
\begin{equation}\label{pr2}U^{(\gamma_{K^*})^*}\uparrow U^{(\gamma_{A^*})^*}\text{ \ pointwise on $\mathbb R^n$ \ (as\ }K^*\uparrow A^*).\end{equation}
Also observe that $(\gamma_{K^*})^*\in\mathcal E_K^+$, which is seen from (\ref{K}) with $\mu=\nu=\gamma_{K^*}$.

We begin by establishing
\begin{equation}\label{baleqK}(\gamma_{K^*})^*=\varepsilon_y^K\text{ \  for every\ }K\in\mathfrak C_A.\end{equation}
Combining (\ref{KP}), applied to $\gamma_{K^*}$, with (\ref{eq-ex0}), applied to $K^*$, gives\footnote{Here we have used the fact that for any $E\subset\mathbb R^n$, $c_\alpha(E)=0$ if and only if $c_\alpha(E^*)=0$, where $E^*$ is the $J_y$-im\-age of $E\setminus\{y\}$ \cite[p.~261]{L}. This also implies that $\nu^*$ is absolutely continuous whenever $\nu$ is so.\label{foot1}}
\[U^{(\gamma_{K^*})^*}(x)=|x^*-y|^{n-\alpha}U^{\gamma_{K^*}}(x^*)=|x^*-y|^{n-\alpha}=|x-y|^{\alpha-n}=U^{\varepsilon_y}(x)\text{ \ n.e.\ on\ }K,\]
which is (\ref{eq-char}) with $\xi=(\gamma_{K^*})^*$ and $\mu=\varepsilon_y$. According to Theorem~\ref{l-char}, (\ref{baleqK}) will follow once we verify that $U^\xi\geqslant U^{(\gamma_{K^*})^*}$ on $\mathbb R^n$ for any $\xi\in\mathfrak M^+$ with
\[U^\xi\geqslant U^{\varepsilon_y}\text{ \ n.e.\ on\ }K.\]
As seen from the last two displays, the inequality in question holds, indeed, n.e.\ on $K$, hence $(\gamma_{K^*})^*$-a.e.\ because $(\gamma_{K^*})^*\in\mathcal E_K^+$, and therefore, by the domination principle, on all of $\mathbb R^n$, as required.

But, according to Theorem~\ref{incr} with $\mu=\varepsilon_y$,
\[U^{\varepsilon_y^K}\uparrow U^{\varepsilon_y^A}\text{ \ pointwise on $\mathbb R^n$ \ (as\ } K\uparrow A).\]
Having now substituted (\ref{baleqK}) into this display, and then compared the relation thus obtained with (\ref{pr2}), we get (\ref{baleq}).\end{proof}

\begin{lemma}\label{l3}If\/ $\gamma_{A^*}$ does not exist, then\/
\[\varepsilon_y^A=\varepsilon_y,\] and hence\/ $y$ is inner\/ $\alpha$-reg\-ular for\/ $A$.\end{lemma}

\begin{proof}  Assuming $\gamma_{A^*}$ does not exist, we begin by observing that then $\varepsilon_y^A(\{y\})>0$. Indeed, if not, then the Kelvin transform $(\varepsilon_y^A)^*=\mathcal K_y\varepsilon_y^A$ exists and has the potential equal to $1$ n.e.\ on $A^*$, the latter being seen by applying (\ref{eq-bala-f1}) and (\ref{KP}) to $\varepsilon_y$ and $\varepsilon_y^A$, respectively. Hence, $(\varepsilon_y^A)^*\in\Theta_{A^*}$, which by Lemma~\ref{l-eq-ar} (applied to $A^*$) contradicts our assumption.

We proceed by showing that the relation $\varepsilon_y^A(\{y\})>0$ thus obtained implies $\varepsilon_y^A=\varepsilon_y$.
Indeed, if not, then
\[\varepsilon_y^A=c\varepsilon_y+\chi,\]
where $0<c<1$ and $\chi\in\mathfrak M^+$ is a nonzero measure with $\chi(\{y\})=0$, the inequality $c<1$ being clear from (\ref{t-mass}) applied to $\varepsilon_y$. But then
\[|x-y|^{\alpha-n}=U^{\varepsilon_y}(x)=U^{\varepsilon_y^A}(x)=c|x-y|^{\alpha-n}+U^\chi(x)\text{ \ n.e.\ on\ }A,\]
and consequently
\[U^{\chi_1}(x)=|x-y|^{\alpha-n}\text{ \ n.e.\ on $A$},\text{ \ where\ }\chi_1:=\chi/(1-c).\]
Since $\chi_1(\{y\})=0$, (\ref{KP}) applied to $\chi_1$ yields $U^{(\chi_1)^*}=1$ n.e.\ on $A^*$, and hence $(\chi_1)^*\in\Theta_{A^*}$, which in view of Lemma~\ref{l-eq-ar} again contradicts our hypothesis that $\gamma_{A^*}$ does not exist.\end{proof}

\begin{theorem}\label{nands}A point\/ $y\in\overline{A}$ is inner\/ $\alpha$-irregular for\/ $A$ if and only if\/ $\gamma_{A^*}$ exists, where\/ $A^*$ is the\/ $J_y$-im\-age of\/ $A\setminus\{y\}$. Moreover, then\/ {\rm(\ref{baleq})} holds.\end{theorem}

\begin{proof}This follows by combining Lemmas~\ref{l2} and \ref{l3}.\end{proof}

\subsection{Proof of Theorem~\ref{irreg-com}}\label{proof} We shall first establish the following lemma.

\begin{lemma}\label{reg-com} If\/ $\gamma_A$ exists, then
\[U^{\gamma_A}=1\text{ \ on\ }\overline{A}\setminus A_I.\]
\end{lemma}

\begin{proof}Fix $x\in\overline{A}\setminus A_I$ and $y\in\mathbb R^n$, $y\ne x$, and write $r:=|x-y|$. For $E\subset\mathbb R^n$, let $E^*$ denote the $J_y$-im\-age of $E\setminus\{y\}$. Then
\[M^{-1}c_\alpha(E)\leqslant c_\alpha(E^*)\leqslant Mc_\alpha(E)\text{ \ for every\ }E\subset B(x,r/2),\]
$M\in(1,\infty)$ being independent of $E$ (cf.\ (\ref{est})). By the Wiener type criterion, this implies that $x^*:=J_y(x)$ is inner $\alpha$-reg\-ular for $A^*$. Hence, by (\ref{eq-reg1}) applied to $\varepsilon_y$,
\begin{equation}\label{reggg}U^{\varepsilon_y^{A^*}}(x^*)=U^{\varepsilon_y}(x^*)=|x^*-y|^{\alpha-n}.\end{equation}

Assume that $\gamma_A$ exists. According to Lemma~\ref{l2} with $A$ and $A^*$ reversed,
\[\gamma_A=\mathcal K_y\varepsilon_y^{A^*},\]
and $U^{\gamma_A}(x)=1$ is obtained by combining (\ref{reggg}) with (\ref{KP}), applied to $\nu=\varepsilon_y^{A^*}$.\end{proof}

Write $A^{(k)}:=A\cap B(0,k)$. Theorem~\ref{irreg-com} will be established once we show
\begin{equation}\label{kel}c_\alpha\bigl(A_I\cap A^{(k)}\bigr)=0\text{ \ for every\ }k\in\mathbb N.\end{equation}
By Lemma~\ref{reg-com} applied to $A^{(k)}$,
$U^{\gamma_{A^{(k)}}}=1$ on $\overline{A^{(k)}}\setminus A^{(k)}_I$.
Comparing this with (\ref{eq-ex0}), applied to $A^{(k)}$, implies that $A^{(k)}_I\cap A^{(k)}$ has inner capacity zero. But it is clear from the Wiener type criterion that
$A_I\cap A^{(k)}=A^{(k)}_I\cap A^{(k)}$, and (\ref{kel}) follows.

\section{Description of $U^{\gamma_A}$ and $S(\gamma_A)$ for $A$ closed}

Define the \emph{reduced kernel\/} $\breve{A}$ \cite[p.~164]{L} of $A$ as the set of all $x\in A$ such that
\[c_\alpha(B(x,r)\cap A\bigr)>0\text{ \ for any\ }r>0.\]

Assuming that $A$ is closed and $\gamma_A$ exists, we provide in Theorem~\ref{desc-eq} below a detailed description of $U^{\gamma_A}$ and $S(\gamma_A)$. While doing this, we can assume without any loss of generality that $A=\breve{A}$. Indeed, $\breve{A}$ is closed along with $A$. Furthermore, since $c_\alpha(A\setminus\breve{A})=0$, $\gamma_A$ serves simultaneously as $\gamma_{\breve{A}}$, and also $A_I=(\breve{A})_I$.

\begin{lemma}\label{l-comp}Under these hypotheses, assume moreover that\/ $\alpha=2$. Then there is a unique connected component\/ $\Delta$ of the\/ {\rm(}open{\rm)} set\/ $A^c$ such that\/ $\gamma_{\tilde{A}}$ exists, where
\[\tilde{A}:=\Delta^c\quad({}\supset A).\]\end{lemma}

\begin{proof}For $A$ compact, $\Delta$ is, in fact, the (unique) unbounded connected component of $A^c$. For $A$ noncompact, fix any $y\in A^c$ and consider the $J_y$-im\-age $K_y$ of $\mathrm{Cl}_{\overline{\mathbb R^n}}A$. Since $\gamma_A$ exists, Theorem~\ref{nands} with $A$ and $A^*$ reversed shows that $y$ is a $2$-ir\-reg\-ular point of $K_y$. By \cite[Chapter~VIII, Section~6, Remark]{Brelo1}, there is therefore a unique connected component $D$ of the (open) set $K_y^c$ such that $y$ is $2$-ir\-reg\-ular for $D^c$, and the $J_y$-im\-age $\Delta$ of this $D$ is as claimed.\end{proof}

\begin{theorem}\label{desc-eq}In the above mentioned hypotheses and notation, {\rm(\ref{eq-ex0})} and\/ {\rm(\ref{eq-ex1})} can be specified as follows: if\/ $\alpha<2$, then
\begin{align}\label{det1} U^{\gamma_A}&=1\text{ \ on\ }A\setminus A_I,\\
\label{det2} U^{\gamma_A}&<1\text{ \ on\ }A^c,\end{align}
while for\/ $\alpha=2$,\footnote{Observe that $\tilde{A}_I\subset A_I$, which is seen from Corollary~\ref{reg-inner} and the inclusion $\partial\tilde{A}\subset\partial A$.}
\begin{align}\label{det3} U^{\gamma_A}&=1\text{ \ on\ }\tilde{A}\setminus\tilde{A}_I,\\
\label{det4} U^{\gamma_A}&<1\text{ \ on\ }\tilde{A}^c.\end{align}
Furthermore,
\begin{equation}\label{det5}
 S(\gamma_A)=\left\{
\begin{array}{lll} A&\text{if}&\alpha<2,\\
\partial\tilde{A}&\text{if}&\alpha=2.\\ \end{array} \right.
\end{equation}
\end{theorem}

\begin{proof}Assume first that $\alpha<2$. Noting that (\ref{det1}) has been established in Lemma~\ref{reg-com}, we first prove
\begin{equation}\label{lesssup}U^{\gamma_A}<1\text{ \ on\ }S(\gamma_A)^c.\end{equation}
Suppose that this fails for some $x_0\in S(\gamma_A)^c$. Then, according to (\ref{eq-ex1}),
\begin{equation}\label{one}U^{\gamma_A}(x_0)=1.\end{equation}
Choose $\varepsilon>0$ so that $\overline{B}(x_0,\varepsilon)\subset S(\gamma_A)^c$. Since
$U^{\gamma_A}$ is $\alpha$-har\-monic on $B(x_0,\varepsilon)$ \cite[Chapter~I, Section~6, n$^\circ$~20]{L} and continuous on $\overline{B}(x_0,\varepsilon)$, we conclude
from (\ref{eq-ex1}) and (\ref{one}) with the aid of \cite[Theorem~1.28]{L} that
$U^{\gamma_A}=1$ a.e.\ on $\mathbb R^n$. By the definition of $\alpha$-super\-har\-mon\-icity, this yields $U^{\gamma_A}=1$ on $\mathbb R^n$. Hence, $\gamma_A$ serves as an inner
equilibrium measure on the whole of $\mathbb R^n$, which is impossible (e.g.\ by Theorem~\ref{l-eq-ex}).

To prove the former equality in (\ref{det5}), suppose to the contrary that there is
$x_1\in A$ such that $x_1\not\in S(\gamma_A)$, and consider an open neighborhood $V\subset S(\gamma_A)^c$
of $x_1$. In view of (\ref{lesssup}), $U^{\gamma_A}<1$ on $V$. On the other hand, since $c_\alpha(V\cap A)>0$ because of the assumption $A=\breve{A}$, we see from (\ref{eq-ex0}) that $U^{\gamma_A}(x_2)=1$ for some $x_2\in V\cap A$. The contradiction obtained shows that, indeed, $S(\gamma_A)=A$.
Substituting this equality into (\ref{lesssup}) establishes (\ref{det2}).

In the rest of the proof, $\alpha=2$. We first establish (\ref{det4}) and the latter relation in (\ref{det5}) for $\gamma_{\tilde{A}}$ in place of $\gamma_A$.
(Relation (\ref{det3}) with $\gamma_{\tilde{A}}$ in place of $\gamma_A$ holds according to Lemma~\ref{reg-com}.)
Suppose that (\ref{det4}) fails
for some $x_3\in\tilde{A}^c$. By (\ref{eq-ex1}) applied to $\tilde{A}$, then the function $U^{\gamma_{\tilde{A}}}$ takes its maximum value $1$ at $x_3$, and hence everywhere on $\tilde{A}^c$, $U^{\gamma_{\tilde{A}}}$ being harmonic on the domain $\tilde{A}^c$.
This combined with (\ref{eq-ex0}) gives $U^{\gamma_{\tilde{A}}}=1$ n.e.\ on $\mathbb R^n$, which is impossible (e.g.\ by Theorem~\ref{l-eq-ex}).

By use of \cite[Theorem~1.12]{L}, we observe from (\ref{eq-ex0}) applied to $\tilde{A}$ that the restriction of $\gamma_{\tilde{A}}$ to the interior of $\tilde{A}$ equals $0$, and so $S(\gamma_{\tilde{A}})\subset\partial\tilde{A}$.
For the converse, suppose to the contrary that there is $x_4\in\partial\tilde{A}$ such that $x_4\notin S(\gamma_{\tilde{A}})$. Choose an open neighborhood $V_1$ of $x_4$ so that $V_1\cap S(\gamma_{\tilde{A}})=\varnothing$. Since $c_2(V_1\cap\tilde{A})>0$, $U^{\gamma_{\tilde{A}}}$ takes the value $1$ at some point in $V_1$, and hence everywhere on $V_1$, again by the maximum principle. This contradicts (\ref{det4}).

The proof is completed by noting that $\gamma_A=\gamma_{\tilde{A}}$. Indeed, as $\partial\tilde{A}\subset A\subset\tilde{A}$, both $\gamma_A$ and $\gamma_{\tilde{A}}$ are supported by $A$ and have the potentials equal to $1$ n.e.\ on $A$. Being absolutely continuous by Lemma~\ref{l-eq-ar}, these measures must be equal according to Corollary~\ref{clos-eq}.
\end{proof}

\section{Integral representation of inner swept measure and applications}\label{onemore}

Throughout this section we assume that
\[\Omega:=\overline{A}^{\,c}\ne\varnothing.\]

\subsection{Integral representation of inner swept measure}\label{sec-repr}
Lemma~\ref{ad} and Theorem~\ref{th-int-rep} below strengthen \cite[Lemma~3.16, Theorem~3.17]{FZ}, dealing with balayage onto closed sets. For the notion of a $\mu$\emph{-ad\-equ\-ate\/} family of measures, see \cite[Chapter~V, Section~3, n$^\circ$~1, Definition~1]{B2}.

\begin{lemma}\label{ad}For every\/ $\mu\in\mathfrak M^+_\Omega$, the family\/ $(\varepsilon_y^A)_{y\in\Omega}$ is\/ $\mu$-adeq\-uate; that is,
\begin{itemize}
\item[\rm{(a)}] for any\/ $f\in C_0(\mathbb R^n)$, the function\/ $y\mapsto\int f\,d\varepsilon_y^A$ is\/ $\mu$-in\-tegrable on\/ $\Omega$,
\item[\rm{(b)}] the map\/ $y\mapsto\varepsilon_y^A$ is vaguely\/ $\mu$-measurable on\/ $\Omega$.
\end{itemize}\end{lemma}

\begin{proof}The proof of (a) repeats word-by-word the proof of Lemma~3.16(a) in \cite{FZ}, except for applying our relations (\ref{alternative}) and (\ref{t-mass}) instead of \cite[Eqs.~(3.11), (3.18)]{FZ}, the concept of balayage being now understood in a way described in Section~\ref{sec-bal} above.
Since $U^{\mu^A}$ is finitely continuous on $\Omega$, the inner swept measure $\mu^A$ being supported by $\overline{A}$, (b) can likewise be obtained by an adaptation of \cite[Proof of Lemma~3.16(b)]{FZ}.\end{proof}

\begin{theorem}\label{th-int-rep}For any $\mu\in\mathfrak M^+_\Omega$, the integral representation\/ {\rm(\ref{L-repr})} holds.
\end{theorem}

\begin{proof}
Fix $\mu\in\mathfrak M^+_\Omega$. Since the family $(\varepsilon_y^A)_{y\in\Omega}$ is $\mu$-ad\-equ\-ate, we can according to \cite[Chapter~V, Section~3, n$^\circ$~2]{B2} define the Radon measure $\nu=\int\varepsilon_y^A\,d\mu(y)$ on $\mathbb R^n$ by means of the formula
\[
\int f(z)\,d\nu(z)=\int\biggl(\int f(z)\,d\varepsilon_y^A(z)\biggr)\,d\mu(y)\text{ \ for every\ }f\in C_0(\mathbb R^n).
\]
According to \cite[Chapter~V, Section~3, Proposition~1]{B2}, this identity remains valid when $f$ is allowed to be any positive l.s.c.\ function on $\mathbb R^n$. For given $x\in\mathbb R^n$ we apply this to $f(z)=|x-z|^{\alpha-n}$, $z\in\mathbb R^n$:
\begin{equation}\label{repr-th1}
U^\nu(x)=\int\biggl(\int|x-z|^{\alpha-n}\,d\varepsilon_y^A(z)\biggr)\,d\mu(y)=\int U^{\varepsilon_y^A}(x)\,d\mu(y).
\end{equation}

To establish (\ref{L-repr}), it remains to show that $\nu=\mu^A$, or equivalently (cf.\ Definition~\ref{def-bal})
\[E(\nu,\lambda)=E(\mu,\lambda^A)\text{ \ for every\ }\lambda\in\mathcal E^+.\]
Applying (\ref{alternative}) with $\mu=\varepsilon_y$ and (\ref{repr-th1}), we get by Fubini's theorem
\begin{align*}E(\nu,\lambda)&=\int U^\nu(x)\,d\lambda(x)=\int\biggl(\int U^{\varepsilon_y^A}(x)\,d\mu(y)\biggr)\,d\lambda(x)\\
  {}&={\int\biggl(\int U^{\varepsilon_y^A}(x)\,d\lambda(x)\biggr)\,d\mu(y)
  =\int\biggl(\int U^{\varepsilon_y}(x)\,d\lambda^A(x)\biggr)\,d\mu(y)}\\
  {}&=\int\biggl(\int|x-y|^{\alpha-n}\,d\mu(y)\biggr)\,d\lambda^A(x)
  =\int U^\mu\,d\lambda^A=E(\mu,\lambda^A),
\end{align*}
which is the required identity.\end{proof}

\begin{corollary}\label{C}For any\/ $\mu\in\mathfrak M^+_\Omega$, $\mu^A$ is absolutely continuous.
\end{corollary}

\begin{proof}Consider a compact set $K\subset\overline{A}$ with $c_\alpha(K)=0$. For every $y\in\Omega$, the swept measure $\varepsilon_y^A$ has finite energy (see Remark~\ref{en-fin}), and hence $\varepsilon_y^A(K)=0$. Applying \cite[Chapter~V, Section~3, Theorem~1]{B2}, we obtain by (\ref{L-repr})
\[\int 1_K\,d\mu^A=\int\,d\mu(y)\int 1_K(x)\,d\varepsilon_y^A(x)=0,\]
and hence $\mu^A$ is indeed absolutely continuous.\end{proof}

\begin{corollary}\label{unicity}For any\/ $\mu\in\mathfrak M^+_\Omega$ and\/ $A$ closed, $\mu^A$ is determined uniquely by\/ {\rm(\ref{eq-bala-f1})} among the absolutely continuous measures of the class\/ $\mathfrak M^+_A$.\end{corollary}

\begin{proof}This follows from Corollary~\ref{C} in view of \cite[p.~178, Remark]{L}.\end{proof}

\subsection{Description of support of inner swept measure}\label{sec-desc} Let $A$ be closed and $\mu\in\mathfrak M^+_\Omega$.
To establish a description of $S(\mu^A)$, we assume that $A$ coincides with its reduced kernel $\breve{A}$, while $\mu$ is carried by a connected component $\Omega_0$ of $\Omega$. This involves no loss of generality, which is seen from (\ref{lin}) and (\ref{redbal}) (with $A'=\breve{A}$).

\begin{theorem}\label{desc-sup} Under these hypotheses,
\begin{equation}\label{lemma-desc-riesz}
S(\mu^A)=\left\{
\begin{array}{lll} A & \text{if} & \alpha<2,\\ \partial\Omega_0  & \text{if} & \alpha=2.\\
\end{array} \right.
\end{equation}
\end{theorem}

\begin{proof}For any $y\in\Omega_0$, we denote $K_y$ the $J_y$-image of $\mathrm{Cl}_{\overline{\mathbb R^n}}A$, and $\gamma_{K_y}$ the equilibrium measure on the (compact) set $K_y$. Since the $J_y$-im\-age of any $E\subset A$ with $c_\alpha(E)=0$ has again zero inner capacity (cf.\ footnote~\ref{foot1}), $K_y$ coincides with its reduced kernel. Thus, $S(\gamma_{K_y})=K_y$ for $\alpha<2$, while for $\alpha=2$, $S(\gamma_{K_y})$ coincides with the \emph{outer boundary\/} of $K_y$, that is, the boundary of the unbounded connected component of $K_y^c$ (see (\ref{det5}) or \cite[Chapter~II, Section~3, n$^\circ$~13]{L}). Having now applied the integral representation (\ref{L-repr}) (which holds under the stated hypotheses, see Theorem~\ref{th-int-rep}) and the fact that for every $y\in\Omega_0$, $\varepsilon_y^A$ is the Kelvin transform of $\gamma_{K_y}$ (see Lemma~\ref{l2}), we obtain (\ref{lemma-desc-riesz}).\end{proof}

\subsection{Further criteria for the existence of inner equilibrium measure}\label{fur-cr} Finally, we provide necessary and sufficient conditions for the existence of $\gamma_A$, given in terms of $\mu^A(\mathbb R^n)$ with $\mu$ suitably chosen.
Since $\gamma_A$ and $\gamma_{\breve{A}}$ exist or do not exist simultaneously, we can certainly assume that $A=\breve{A}$.

\begin{theorem}\label{bal-mass-th} $\gamma_A$ exists if there is a measure\/ $\mu\in\mathfrak M^+_\Omega$ with
\begin{equation}\label{mass-measure}\mu^A(\mathbb R^n)<\mu(\mathbb R^n).\end{equation}
\end{theorem}

\begin{proof} Assume that (\ref{mass-measure}) holds for some $\mu\in\mathfrak M^+_\Omega$, and suppose, to the contrary, $\gamma_A$ does not exist. Fix $y\in\Omega$ and consider the $J_y$-im\-age $A^*$ of $A$; then, according to Lemma~\ref{l3} with $A$ and $A^*$ reversed, $y$ is inner $\alpha$-reg\-ular for $A^*$. According to Lemma~\ref{reg-com}, this gives
\begin{equation}\label{nuu}U^{\gamma_{A^*}}(y)=1.\end{equation}
(Note that an inner equilibrium measure $\gamma_{A^*}$ exists, $A^*$ being relatively compact.)
On the other hand, for the Kelvin transform $(\gamma_{A^*})^*=\mathcal K_y\gamma_{A^*}$ of $\gamma_{A^*}$, we have according to Lemma~\ref{l2}
\[(\gamma_{A^*})^*=\varepsilon_y^A.\]
Therefore, having applied (\ref{kelv-mmm}) to $\nu=\varepsilon_y^A$, we get in view of (\ref{nuu})
\begin{equation*}\label{1}\varepsilon_y^A(\mathbb R^n)=U^{\gamma_{A^*}}(y)=1.\end{equation*}
Substituting this now into (\ref{L-repr}) (which holds according to Theorem~\ref{th-int-rep}) and applying \cite[Chapter~V, Section~3, Theorem~1]{B2}, we obtain
\[\mu^A(\mathbb R^n)=\int 1\,d\mu^A=\int d\mu(y)\int 1(x)\,d\varepsilon_y^A(x)=\int 1\,d\mu=\mu(\mathbb R^n),\]
which, however, contradicts (\ref{mass-measure}).\end{proof}

For $A$ closed, Theorem~\ref{bal-mass-th} can be reversed.

\begin{theorem}\label{bal-mass-th'}
For\/ $A$ closed, $\gamma_A$ exists if and only if\/ {\rm(\ref{mass-measure})} holds for some\/ $\mu\in\mathfrak M^+_\Omega$. Actually, if\/ $\gamma_A$ exists, then\/ {\rm(\ref{mass-measure})} does hold for every nonzero\/ $\mu\in\mathfrak M^+_{\Omega_\alpha}$, where
\[\Omega_\alpha:=\left\{
\begin{array}{lll} \Omega & \text{if} & \alpha<2,\\ \Delta  & \text{if} & \alpha=2,\\
\end{array} \right.
\]
$\Delta$ being defined in Lemma\/~{\rm\ref{l-comp}}.

\end{theorem}

\begin{proof}
In view of Theorem~\ref{bal-mass-th}, it is enough to establish the latter part of the theorem. Assume $\gamma_A$ exists and fix any nonzero $\mu\in\mathfrak M^+_{\Omega_\alpha}$. By (\ref{det2}) and (\ref{det4}),
\begin{equation}\label{strless}
U^{\gamma_A}<1\text{ \ on\ }\Omega_\alpha.
\end{equation}
Since both $\gamma_A$ and $\mu^A$ are absolutely continuous (see Lemma~\ref{l-eq-ar} and Corollary~\ref{C}, respectively) and supported by $A$, we obtain from (\ref{eq-ex0}), (\ref{eq-bala-f1}), and (\ref{strless})
\[\mu^A(\mathbb R^n)=\int U^{\gamma_A}\,d\mu^A=\int U^\mu\,d\gamma_A=\int U^{\gamma_A}\,d\mu<\mu(\Omega_\alpha)=\mu(\mathbb R^n),\]
which is (\ref{mass-measure}).\end{proof}

\begin{example}\label{ex2}Let $n=3$, $\alpha=2$, and let $A$ be a rotation body defined by (\ref{defA}). Then
\[\mu^A(\mathbb R^3)=\mu(\mathbb R^3)\text{ \ for every\ }\mu\in\mathfrak M^+_{A^c},\]
provided that $\varrho$ in (\ref{defA}) is given by (\ref{c1}), and
\[\mu^A(\mathbb R^3)<\mu(\mathbb R^3)\text{ \ for every nonzero\ }\mu\in\mathfrak M^+_{A^c}\]
whenever (\ref{c2}) or (\ref{c3}) holds (see Figure~\ref{Fig1}). This follows by combining Theorems~\ref{bal-mass-th}, \ref{bal-mass-th'} and Corollary~\ref{ineq-en} with Example~\ref{ex-thin}.
\end{example}

\section{Acknowledgement} The author thanks Bent Fuglede for many fruitful discussions on the topic of the paper.

\end{document}